\documentclass[11pt,a4paper]{article}
\usepackage[margin=25mm]{geometry}
\usepackage[T1]{fontenc}
\usepackage{lmodern,microtype}
\usepackage{amsmath,amssymb,amsthm,mathtools,mathrsfs}
\usepackage{enumitem,booktabs,longtable,array,aliascnt,xcolor}
\usepackage{indentfirst}
\usepackage{hyperref}
\usepackage[nameinlink,capitalise,noabbrev]{cleveref}
\makeatletter
\AtBeginDocument{%
  \let\paper@hyper@refstepcounter\hyper@refstepcounter
  \def\paper@thm@refstepcounter#1{%
    \HyCnt@ProvideTheHCounter{#1}%
    \hyper@makecurrent{#1}%
  }%
  \let\paper@thm\@thm
  \def\@thm{%
    \let\hyper@refstepcounter\paper@thm@refstepcounter
    \paper@thm
  }%
  \let\paper@begintheorem\@begintheorem
  \def\@begintheorem#1#2[#3]{%
    \let\hyper@refstepcounter\paper@hyper@refstepcounter
    \paper@begintheorem{#1}{#2}[#3]%
  }%
  \let\paper@deferred@thm@head\deferred@thm@head
  \def\deferred@thm@head#1{%
    \paper@deferred@thm@head{%
      \Hy@raisedlink{\hyper@anchorstart{\@currentHref}\hyper@anchorend}%
      #1}%
  }%
}
\makeatother
\hypersetup{hidelinks,pdftitle={One-dimensional pressureless flow with a nonlinear velocity law: Bregman variational construction and well-posedness},pdfauthor={Gaowei Cao, Feimin Huang, Guirong Tang}}
\numberwithin{equation}{section}
\allowdisplaybreaks[2]
\makeatletter
\renewcommand*\l@subsection{\@dottedtocline{2}{2em}{2.8em}}
\makeatother
\setlist[enumerate]{leftmargin=2.5em,itemsep=0.2em,font=\normalfont}
\newtheorem{theorem}{Theorem}[section]
\newaliascnt{lemma}{theorem}\newtheorem{lemma}[lemma]{Lemma}\aliascntresetthe{lemma}
\newaliascnt{proposition}{theorem}\newtheorem{proposition}[proposition]{Proposition}\aliascntresetthe{proposition}
\newaliascnt{corollary}{theorem}\newtheorem{corollary}[corollary]{Corollary}\aliascntresetthe{corollary}
\theoremstyle{definition}
\newaliascnt{definition}{theorem}\newtheorem{definition}[definition]{Definition}\aliascntresetthe{definition}
\newaliascnt{example}{theorem}\newtheorem{example}[example]{Example}\aliascntresetthe{example}
\theoremstyle{remark}
\newaliascnt{remark}{theorem}\newtheorem{remark}[remark]{Remark}\aliascntresetthe{remark}
\newcommand{\R}{\mathbb R}
\newcommand{\K}{\mathcal K}

\newcommand{\M}{\mathcal M}
\newcommand{\dd}{\,\mathrm d}
\newcommand{\one}{\mathbf1}
\newcommand{\supp}{\operatorname{supp}}
\newcommand{\TV}{\operatorname{TV}}
\newcommand{\osc}{\operatorname{osc}}
\newcommand{\esssup}{\operatorname*{ess\,sup}}

\newcommand{\W}{\mathfrak W_{f_R}}
\newcommand{\qmeas}{\mathfrak q}
\newcommand{\mean}[2]{\frac{1}{|#1|}\int_{#1}#2\dd a}

\newcommand{\ip}[2]{\left\langle #1,#2\right\rangle}

\title{One-dimensional pressureless flow with a nonlinear velocity law: Bregman variational construction and well-posedness}

\author{Gaowei Cao\thanks{\raggedright Hetao Institute of Mathematics and Interdisciplinary Sciences, Shenzhen 518038, China. E-mail: \texttt{caogaowei@himis-sz.cn}.}
\and Feimin Huang\thanks{\raggedright Corresponding author. Academy of Mathematics and Systems Science, Chinese Academy of Sciences, Beijing 100190, China; School of Mathematical Sciences, University of Chinese Academy of Sciences, Beijing 100049, China. E-mail: \texttt{fhuang@amt.ac.cn}.}
\and Guirong Tang\thanks{\raggedright School of Mathematical Sciences, Capital Normal University, Beijing 100048, China. E-mail: \texttt{tangguirong@amss.ac.cn}.}}

\date{}

\begin{document}
\maketitle
\begin{abstract}
We establish global existence, uniqueness and quantitative stability for one-dimensional pressureless flow governed by
\[
\partial_t\rho+\partial_x\bigl(\rho f(u)\bigr)=0,
\qquad
\partial_t(\rho u)+\partial_x\bigl(\rho u f(u)\bigr)=0,
\]
where $\rho$ is the mass density, possibly measure-valued, $u$ is the physical velocity, and $f\in C^1(\mathbb R)$, with $f'>0$, determines the transport velocity $f(u)$. The initial mass has finite positive mass and finite second moment, allowing atoms and unbounded support, and the initial physical velocity is bounded.
The difficulty at a cluster collision is that mass and momentum conservation determine a mass-weighted average of physical velocity, whereas transport is governed by the nonlinear law $f$ applied to that average. We construct global entropy solutions by a Bregman variational time discretization whose cost is chosen to reproduce the transport velocity determined by this mass-weighted average, coupled with conditional-expectation updates of physical velocity.
For every entropy solution satisfying the Oleinik condition on the transport velocity and initial weak continuity of the kinetic energy measure, we establish a Lagrangian representation and prove uniqueness within this class by comparing flows with possibly different cluster intervals.
For two such solutions with the same velocity law and common mass, we obtain a $W_\infty$ estimate for the normalized mass measures in terms of the initial position and physical-velocity discrepancies in mass coordinates, whenever the initial $W_\infty$ distance is finite. The Lagrangian evolution has the semigroup property.
\end{abstract}

\tableofcontents

\section{Introduction}\label{sec:intro}

We study the Cauchy problem for mass and momentum transported by a common nonlinear velocity law:
\begin{equation}\label{eq:system}
 \partial_t\rho+\partial_x(\rho f(u))=0,\qquad
 \partial_t(\rho u)+\partial_x(\rho u f(u))=0
 \quad\text{in }\R\times(0,\infty).
\end{equation}
Here $\rho$ is the mass density, possibly measure-valued, and $u$ is the physical velocity; both conserved quantities move with transport velocity $f(u)$. We assume
\begin{equation}\label{eq:velocity-law}
 f\in C^1(\R),\qquad f'(r)>0
 \quad\text{for every }r\in\R.
\end{equation}
Thus $f$ is strictly increasing. In smooth regions of positive density the second equation reduces, using the first, to $u_t+f(u)u_x=0$, so $f(u)$ is also the characteristic speed.

The choice $f(u)=u$ recovers the classical pressureless Euler system. Nonlinear laws satisfying \eqref{eq:velocity-law} occur in physically motivated transport models. Examples include, after suitable changes of dependent variables, one-dimensional reductions of single-phase moment models in geometrical optics \cite{Runborg} and the baryon-number--momentum formulation of pressureless relativistic flow \cite{YangZhang}.

Compression can concentrate mass at a single position, even when the initial mass has a density. A solution must therefore allow mass measures with atoms. An atom represents a cluster, and clusters may merge when they meet. Conservation of mass and momentum determines the physical velocity after a merger: clusters with masses $m_j>0$ and physical velocities $u_j$ merge with physical velocity $p=\sum_jm_ju_j/\sum_jm_j$. Their subsequent transport velocity is
\[
 f(p)=f\!\left(\frac{\sum_jm_ju_j}{\sum_jm_j}\right)
 \ne\frac{\sum_jm_jf(u_j)}{\sum_jm_j}
 \quad\text{in general}.
\]
The order of averaging and applying $f$ is therefore essential. A variational construction must recover the velocity prescribed by conservation while permitting mass to concentrate and clusters to merge.

Our aim is to construct global entropy solutions for measure-valued initial mass, recover a Lagrangian representation for every solution in the chosen Eulerian class, and prove uniqueness and quantitative stability by comparing the resulting flows.

\subsection{Previous work and contribution}\label{sec:previous-work}

For the nonlinear system \eqref{eq:system}, Huang \cite{Huang} proved global entropy existence for bounded nonnegative initial density and bounded initial physical velocity, under the hypotheses on the increasing velocity law used there. The construction uses piecewise constant Riemann solutions, wave-interaction estimates and generalized characteristics, and already incorporates mass-weighted averaging of physical velocity and characteristic triangles. Huang's admissibility condition is an Oleinik estimate on the physical velocity with a data-dependent constant; the constructed solution also satisfies the Oleinik condition on the transport velocity with constant $1/t$.

Our data may contain initial atoms, but have finite mass and finite second moment. Huang's initial density class permits infinite mass and excludes atoms, so neither initial-data class contains the other. We prescribe the Oleinik condition on the transport velocity with the universal constant $1/t$, together with initial weak continuity of the kinetic energy measure. These conditions define the class in which we prove uniqueness; their precise formulation is given in \cref{def:entropy}.

For classical pressureless flow, the generalized variational principle of E, Rykov and Sinai \cite{ERS} yields global weak solutions, and the work of Brenier and Grenier \cite{BG} relates sticky dynamics to scalar conservation laws. Bouchut and James developed duality solutions for transport equations with one-sided Lipschitz coefficients \cite{BJtransport} and applied them to pressureless gases \cite{BJ}. Huang and Wang \cite{HW} proved existence and uniqueness for locally Radon initial mass and bounded measurable initial velocity, imposing an Oleinik condition on the physical velocity and initial weak continuity of the kinetic energy measure. In that setting physical and transport velocities coincide. Thus neither measure data nor the initial energy condition is new to the classical theory.

The classical Lagrangian theory also provides conditional averaging, monotone maps and evolution semigroups. For bounded initial velocities satisfying a negative-jump condition, Moutsinga \cite{Moutsinga} described sticky dynamics through convex hulls, conditional expectations and a forward flow. Natile and Savar\'e \cite{NS} constructed an evolution semigroup for probability measures with finite second moment and square-integrable initial velocity. Brenier, Gangbo, Savar\'e and Westdickenberg \cite{BGSW} used monotone-cone differential inclusions to treat interactions, while Cavalletti, Sedjro and Westdickenberg \cite{CSW} gave a direct projection proof of global weak existence.

Hynd \cite{Hynd} obtained a conditional-expectation Lagrangian evolution for absolutely continuous initial velocity. For finite-second-moment probability data and bounded right-continuous initial velocity, Suder and Tudorascu \cite{ST} proved uniqueness within their sticky particle flow class and identified the Eulerian solution with the scalar-conservation-law solution. Entropy--Lagrangian equivalence with force is developed in \cite{CG}. These results involve different velocity assumptions and solution classes; finite second moments and atomic masses alone do not distinguish our setting.

Variational time discretization is available as well. For classical Euler flow, Cavalletti, Sedjro and Westdickenberg \cite{CSWtime} minimize an acceleration cost over monotone transport maps and project the updated velocity. Their one-dimensional pressureless construction recovers sticky dynamics and distinguishes displacement from updated velocities. For the nonlinear law considered here, we select a Bregman cost whose collision barycenter is the transport velocity of the averaged physical velocity.

Our analysis connects this variational construction with two further steps: recovery of the same nonlinear Lagrangian evolution for every entropy solution in the prescribed Eulerian class, and quantitative comparison of flows whose cluster intervals need not coincide. Together these steps establish global well-posedness and stability for the data and solution class stated below.

\subsection[Main results]{Main results}\label{sec:main-results}

Let $\rho_0\in\M_+(\R)$ and $u_0\in L^\infty(\rho_0)$ be the initial mass measure and physical velocity. We assume that
\begin{equation}\label{eq:P2-data}
 0<M:=\rho_0(\R)<\infty,\qquad
 \int_\R |x|^2\dd\rho_0(x)<\infty.
\end{equation}

Fix $R>0$ such that $\|u_0\|_{L^\infty(\rho_0)}\le R$. By \eqref{eq:velocity-law},
\begin{equation}\label{eq:f}
 c_f:=\min_{|r|\le R}f'(r)>0,\qquad C_f:=\max_{|r|\le R}f'(r)<\infty.
\end{equation}
Set $v_-=f(-R)$, $v_+=f(R)$ and $V_*=\max(|v_-|,|v_+|)$. Thus physical velocities in $[-R,R]$ correspond to transport velocities in $[v_-,v_+]$, bounded in magnitude by $V_*$.

For the convex-duality and comparison arguments, extend $f|_{[-R,R]}$ by the tangent lines at its endpoints and denote the extension by $f_R$. It is a $C^1$ increasing bijection of $\R$, with $c_f\le f_R'\le C_f$, and agrees with $f$ throughout the physical velocity range. The equation itself continues to use $f$.

We express the results in both Eulerian and mass coordinates. Let $X(t)$ be the right-continuous quantile of $\rho_t$, so that $X(t)_\#\dd a=\rho_t$ for the mass label $a\in(0,M)$. The position maps belong to
\[
 \K=\{X\in L^2(0,M):X\text{ has a nondecreasing representative}\}.
\]
Labels at the same position form a fibre; a level interval of positive length represents a cluster. For $X\in\K$, the conditional expectation $P_X=\mathbb E[\,\cdot\mid\sigma(X)]$ is the orthogonal projection onto functions of the position. It averages on each positive fibre and leaves a function unchanged almost everywhere outside those fibres. Applied to physical velocity, it implements mass-weighted averaging within a cluster.

Write $X_0=X_{\rho_0}$ and $U_0=u_0\circ X_0$. Then $X_0\in\K\subset L^2(0,M)$, $|U_0|\le R$, and
\begin{equation}\label{eq:data}
 X_{0\#}\dd a=\rho_0,\quad X_{0\#}(U_0\dd a)=u_0\rho_0,
 \quad U_0=P_{X_0}U_0.
\end{equation}
The compatibility $U_0=P_{X_0}U_0$ means that labels in the same initial atom have the same physical velocity. For a cluster represented by a label interval $I$, the nonlinear averaging issue becomes
\begin{equation}\label{eq:mismatch}
 f\left(\mean{I}{U_0}\right)\ne\mean{I}{f(U_0)}
\end{equation}
in general. In particular, the quadratic projection of $X_0+tf(U_0)$ need not select the transport velocity dictated by momentum conservation; \cref{prop:classical-free} identifies the affine restriction of this rule. The Bregman construction used instead is described in \cref{sec:construction-comparison}.

We call $u^2\rho$ the kinetic energy measure, omitting the factor $1/2$. The solution class is selected by two conditions in addition to the weak conservation laws. The Oleinik condition on the transport velocity controls transport at positive times; initial weak continuity of the kinetic energy measure excludes an energy defect at time zero.

\begin{definition}[Entropy solution]\label{def:entropy}
An entropy solution on $[0,T]$ with data $(\rho_0,u_0)$ consists of $\rho\in C([0,T];\M_+(\R)\text{-narrow})$ of mass $M$ and a jointly measurable physical velocity $|u|\le R$, understood modulo $\rho_t\dd t$. For all test functions $\phi,\psi\in C_c^1(\R\times[0,T))$, it satisfies
\begin{align}
 \int_0^T\!\int(\phi_t+f(u)\phi_x)\dd\rho_t\dd t
 +\int\phi(x,0)\dd\rho_0&=0,\label{eq:weak-m}\\
 \int_0^T\!\int(u\psi_t+uf(u)\psi_x)\dd\rho_t\dd t
 +\int\psi(x,0)u_0\dd\rho_0&=0.\label{eq:weak-q}
\end{align}
The transport velocity admits a jointly Borel extension $b$, with $b=f(u)$ for $\rho_t\dd t$-almost every spacetime point, such that for almost every $t>0$ and all $x<y$,
\begin{equation}\label{eq:OSL}
 b(y,t)-b(x,t)\le (y-x)/t.
\end{equation}
On a fixed full-measure set $G$, the kinetic energy measure is weakly continuous at the initial time, in the sense that
\begin{equation}\label{eq:E0}
 \lim_{t\downarrow0,\ t\in G}\int\varphi(x)|u(x,t)|^2\dd\rho_t
 =\int\varphi(x)|u_0|^2\dd\rho_0
 \qquad(\varphi\in C_b(\R)).
\end{equation}
Let $\qmeas$ be the canonical momentum measure supplied by \cref{lem:canonical}: the unique narrowly continuous signed measure family agreeing with $u\rho_t$ for almost every $t$. We choose $G$ so that $u(\cdot,t)\rho_t=\qmeas_t$ for every $t\in G$.
\end{definition}

The two admissibility conditions serve different parts of the argument. On positive-time strips the Oleinik condition determines the transport of mass and momentum. At time zero the momentum trace determines averages on initial fibres, while the energy trace excludes residual variation within them. For this latter step the proof uses \eqref{eq:E0} with $\varphi\equiv1$; the definition retains the full measure trace. The splitting example in \cref{ex:energy-needed} shows why the conserved quantities alone do not suffice. No separate trace or dissipation condition is imposed on $\int f(u)^2\dd\rho$.

For probability measures $\mu,\nu$ on $\R$, we use the extended distance
\[
 W_\infty(\mu,\nu)
 :=\inf_{\pi\in\Pi(\mu,\nu)}\esssup_{(x,y)\sim\pi}|x-y|
 \in[0,\infty],
\]
where $\Pi(\mu,\nu)$ is the set of probability couplings with these marginals. The following theorem combines existence by the variational construction, uniqueness in the Eulerian class, and quantitative comparison in mass coordinates.

\begin{theorem}[Global well-posedness]\label{thm:main}
\textup{(a)} For a velocity law $f$ satisfying \eqref{eq:velocity-law} and initial data $\rho_0\in\M_+(\R)$, $u_0\in L^\infty(\rho_0)$ satisfying \eqref{eq:P2-data}, the Cauchy problem \eqref{eq:system} admits a unique global entropy solution $(\rho,u)$ in the sense of \textup{\cref{def:entropy}}. Here global means that the definition holds on every finite time interval. Writing $\qmeas_t$ for the canonical narrowly continuous representative of $u\rho_t$ determined by the weak momentum equation, equal data determine $\rho_t,\qmeas_t$ at every $t\ge0$ and $u$ modulo $\rho_t\dd t$. Moreover,
\begin{equation}\label{eq:main-moment}
 \left(\int_\R |x|^2\dd\rho_t(x)\right)^{1/2}
 \le\left(\int_\R |x|^2\dd\rho_0(x)\right)^{1/2}
 +\sqrt M\,V_*t\qquad(t\ge0).
\end{equation}

\textup{(b)} For two such solutions $(\rho,u)$ and $(\widetilde\rho,\widetilde u)$ with common mass $M$, velocity law $f$ and physical velocity range $[-R,R]$, set
\[
 \delta_0:=\|u_0\circ X_{\rho_0}
 -\widetilde u_0\circ X_{\widetilde\rho_0}\|_{L^\infty(0,M)}.
\]
If $d_0:=W_\infty(\rho_0/M,\widetilde\rho_0/M)<\infty$, then
\begin{equation}\label{eq:main-mass-stability}
 W_\infty\!\left(\frac{\rho_t}{M},\frac{\widetilde\rho_t}{M}\right)
 \le W_\infty\!\left(\frac{\rho_0}{M},\frac{\widetilde\rho_0}{M}\right)
 +C_ft\,\delta_0\qquad(t\ge0).
\end{equation}

\textup{(c)} For every entropy solution, let $X(t)=X_{\rho_t}$ and $U_0=u_0\circ X_{\rho_0}$. Then
\begin{align}
 p_X(t)&=P_{X(t)}U_0,\qquad \partial_tX=f(p_X),\label{eq:main-state}\\
 \rho_t&=X(t)_\#\dd a,\qquad\qmeas_t=X(t)_\#(U_0\dd a).\label{eq:main-push}
\end{align}
The conditional physical velocity and push-forward identities hold at every time with the canonical representatives; the derivative is taken in $L^2(0,M)$ for almost every time. Positive fibres have the no-splitting property. For equal initial data, the two Lagrangian maps agree on one common full-measure label set for all $t\ge0$.
\end{theorem}

\begin{remark}[Localization of the velocity-law assumptions]\label{rem:local-law}
For fixed initial data as in \cref{thm:main} and $R$ chosen as above, the conclusions require only a strictly increasing $f\in C^1([-R,R])$ with $f'>0$ on $[-R,R]$. The construction preserves $|u|\le R$, and this bound is part of \cref{def:entropy}; thus the weak conservation laws and entropy conditions depend only on this restriction of $f$. The auxiliary tangent-line extension supplies the global bounds used in the proof. For the comparison in \cref{thm:main}\textup{(b)}, the same interval is chosen to contain both initial physical velocity ranges.
\end{remark}

In part \textup{(b)}, $\delta_0=\|U_0-W_0\|_\infty$, with $W_0=\widetilde u_0\circ X_{\widetilde\rho_0}$: the initial physical velocities are compared at common mass labels. These compositions depend only on the initial almost-everywhere classes and require no values in vacuum. Finite second moments do not imply finite $d_0$. If $d_0=\infty$, the estimate has only its extended-value meaning, whereas equal-data uniqueness is unaffected.

The uniqueness assertion concerns mass and canonical momentum at every time, and physical velocity modulo $\rho_t\dd t$. It does not prescribe velocity in vacuum or a unique auxiliary extension $b$. The same Lagrangian representation also permits restarting the flow and yields the semigroup in \cref{prop:semigroup}.

\subsection[Construction and proof strategy]{Construction and proof strategy}\label{sec:construction-comparison}

\emph{Choosing the cost from the collision law.} Set $H(u)=\int_0^u f_R(r)\dd r$ and $L=H^*$. For the Bregman divergence
\[
 D_L(v\mid w)=L(v)-L(w)-L'(w)(v-w),
\]
the identity $L'(f(u))=u$ on $[-R,R]$ makes the optimality condition an average of physical velocities. Consequently,
\[
 \operatorname*{arg\,min}_{v\in\R}\sum_jm_jD_L(v\mid f(u_j))
 =f\!\left(\frac{\sum_jm_ju_j}{\sum_jm_j}\right),
\]
for $m_j>0$ and $u_j\in[-R,R]$. This is the transport velocity required by mass and momentum conservation. \Cref{prop:generator} characterizes the cost within the separable Bregman class.

\emph{Passing from discrete minimizers to a flow.} The scheme minimizes the Bregman cost of the position increment and then updates physical velocity by conditional expectation. With time step $h$, discrete positions $X^n$ and physical velocities $U^n$, the displacement velocity $V^{n+1}=(X^{n+1}-X^n)/h$ need not equal $f(U^{n+1})$. The normal-cone relation and strong monotonicity of $L'$ bound the squared $L^2$ discrepancy by the kinetic energy lost in the update. Multiplication by $h$ and summation make this discrepancy vanish in $L^2((0,T)\times(0,M))$ as $h\downarrow0$, for each finite $T$.

The no-splitting property makes the conditional averages nested. Their energy identity controls time translations, while approximation of the initial physical velocity by $BV$ functions controls translations in the mass label. Together these estimates yield strong compactness of the physical velocities, without a $BV$ assumption on the initial data or a bound on the position range, and permit passage through the nonlinear law $f$ along a subsequence. The limiting conditional-expectation identities first identify the average on each current fibre. Tangency of the limiting monotone path then makes its derivative constant on that fibre for almost every time; applying the inverse velocity law identifies the physical velocity itself as $P_XU_0$. These two steps yield the closure $X_t=f(P_XU_0)$. Eulerian reconstruction, the initial energy trace and the Oleinik extension complete the existence proof.

\emph{Recovering the structure of every entropy solution.} The constructed flow does not by itself establish uniqueness in the Eulerian class. For an arbitrary entropy solution, the weak equations and the Oleinik condition provide a common transport map for mass and momentum on positive-time strips. To reach the initial time, we combine the momentum trace, which fixes fibre averages, with the kinetic energy trace, which excludes residual variation. This recovers the representation in \cref{thm:main}\textup{(c)}, including on initial atoms, without appealing to the discrete construction.

\emph{Comparing flows with different clusters.} Even after this representation, two solutions may average their initial physical velocities over different label intervals. The compression inequalities in \cref{thm:allcut} control this difficulty. At any earlier source time, including zero, the average over a cluster's left portion bounds its physical velocity from above, and the average over its right portion bounds it from below, whenever the portion has positive length. These inequalities allow us to compare flows with different cluster intervals and to control the growth of their position discrepancy by the initial physical-velocity discrepancy $\|U_0-W_0\|_\infty$.

We apply this control to the essential supremum of the position gap. Unbounded initial support prevents a direct reduction to a maximizing label at a finite position. A spatial penalization preserves fibres and confines near-maximizing labels to finite positions. Short-time expansions through first capture by a cluster allow the future near-maximizing labels to vary with the time increment. The resulting estimate, after integration and removal of the penalty, proves \eqref{eq:main-mass-stability} and equal-data uniqueness.

The same Lagrangian structure yields affine backward fans and one-sided cumulative identities, as well as kinetic energy dissipation, convex entropies and the semigroup property. \Cref{sec:preliminaries} provides the common tools, \cref{sec:existence} carries out the construction, and \cref{sec:lifting} recovers the structure of entropy solutions. \Cref{sec:compare} proves the comparison estimate and uniqueness. The energy and semigroup consequences are developed in \cref{sec:consequences}, and \cref{app:measure} contains the measure and representative proofs.

\emph{Notation.} Set $Q_U(a)=\int_0^aU_0(b)\dd b$. Unless a spacetime domain is specified, $L^q$ spaces and norms in mass coordinates are taken over $(0,M)$.

\section{Lagrangian preliminaries}\label{sec:preliminaries}

We collect the monotone-map, conditional-average and regularity tools used in the construction and in the analysis of entropy solutions. Write $g=f_R^{-1}$. Conditional expectations are taken with respect to Lebesgue-completed sigma-algebras.

Clipping $b$ to $[v_-,v_+]$ leaves its values unchanged $\rho_t\dd t$-almost everywhere and preserves the Oleinik condition on the transport velocity, since clipping is nondecreasing and $1$-Lipschitz. We use this bounded extension.

\subsection[\texorpdfstring{Monotone maps, conditional expectations and normal cones}{Monotone maps, conditional expectations and normal cones}]{Monotone maps, conditional expectations and normal cones}\label{sec:fibres}

A \emph{positive fibre} of $X$ is a maximal open level interval of positive length. It represents a cluster in mass coordinates. Endpoints do not affect conditional averages.

\begin{lemma}[Fibrewise conditional expectation]\label{lem:PX}
For $X\in\K\subset L^2(0,M)$, let $\mathscr I_X$ denote its positive fibres and set
\[
 \Omega_X=\bigcup_{I\in\mathscr I_X}I,\qquad
 N_X=(0,M)\setminus\Omega_X.
\]
Then
\[
 \mathcal H_X=L^2(\sigma(X))
 =\bigl\{W\in L^2(0,M):\ \forall I\in\mathscr I_X,\ \exists c_I\in\R,
 \ W=c_I\ \text{a.e. on }I\bigr\}.
\]
For every $U\in L^2(0,M)$, the orthogonal projection onto $\mathcal H_X$ is
\[
 (P_XU)(a)=
 \begin{cases}
 \displaystyle\frac1{|I|}\int_I U(b)\dd b,&a\in I\in\mathscr I_X,\\[0.6em]
 U(a),&a\in N_X,
 \end{cases}
 \qquad\text{for a.e. }a\in(0,M).
\]
Thus the projection acts as the identity almost everywhere on the non-flat set $N_X$.
\end{lemma}
\begin{proof}
There are at most countably many positive fibres. After removing their endpoints, $X$ is a Borel injection on the non-flat set and its inverse on the image is Borel. Hence a Borel function constant on every positive fibre is a function of $X$, modulo null sets. This proves the characterization of $\mathcal H_X$. The stated averages preserve integrals against every such function and therefore give the conditional expectation.
\end{proof}

\begin{lemma}[Approximation by functions of the position]\label{lem:projection-approximation}
Let $X,Y\in\K$ and $W\in\mathcal H_X$. For every $\varphi\in C_c^\infty(\R)$,
\begin{equation}\label{eq:projection-approximation}
 \|P_YW-W\|_2
 \le\|W-\varphi(X)\|_2+\operatorname{Lip}(\varphi)\|Y-X\|_2.
\end{equation}
Consequently, if $X_n\to X$ in $L^2$, then $P_{X_n}W\to W$ in $L^2$ for every $W\in\mathcal H_X$.
\end{lemma}
\begin{proof}
Since $\varphi(Y)\in\mathcal H_Y$, best approximation gives
$\|P_YW-W\|_2\le\|W-\varphi(Y)\|_2$, and the triangle inequality proves \eqref{eq:projection-approximation}.
Write $W=w(X)$ with $w\in L^2(X_\#\dd a)$. Since $X_\#\dd a$ is a finite Radon measure, $C_c^\infty(\R)$ is dense in this space. Choose $\varphi$ with arbitrarily small $\|W-\varphi(X)\|_2$ and then let $n\to\infty$ in \eqref{eq:projection-approximation}.
\end{proof}

Compression is expressed by the normal cone, with the convention
\begin{equation}\label{eq:normaldef}
 N_\K(X)=\{\xi\in L^2:\ip{\xi}{Y-X}\le0\text{ for all }Y\in\K\}.
\end{equation}

\begin{lemma}[Primitive characterization and orthogonality]\label{lem:normal}
For $X\in\K$ and $\xi\in L^2(0,M)$, put $\Xi(a)=\int_0^a\xi$. Then
\begin{equation}\label{eq:normalchar}
 \xi\in N_\K(X)\quad\Longleftrightarrow\quad
 \Xi(M)=0,\quad\Xi\ge0,\quad\int_{(0,M)}\Xi\dd D_aX=0.
\end{equation}
Here $D_aX$ is the nonnegative, locally finite Stieltjes measure on $(0,M)$; for $\Xi\ge0$, its integral is the increasing limit over compact subintervals. Moreover $P_X\xi=0$, so $\ip{\xi}{W}=0$ for all $W\in\mathcal H_X$.
\end{lemma}
\begin{proof}
Every nondecreasing $Z\in L^2(0,M)$ satisfies
\begin{equation}\label{eq:monotone-endpoint-decay}
 \sqrt a\,|Z(a)|\to0\quad(a\downarrow0),\qquad
 \sqrt{M-a}\,|Z(a)|\to0\quad(a\uparrow M).
\end{equation}
Indeed, if $Z(a)\le0$ then $a|Z(a)|^2\le\int_0^a|Z|^2$, while its positive part is bounded near zero by an interior value. The other endpoint follows by reversing the interval and the sign.
When $\Xi(0)=\Xi(M)=0$, Cauchy--Schwarz gives
\[
 |\Xi(a)|\le\min\{\sqrt a\|\xi\|_{L^2(0,a)},
                         \sqrt{M-a}\|\xi\|_{L^2(a,M)}\},
\]
so $Z\Xi$ vanishes at both endpoints. Integrating by parts on compact intervals with continuity endpoints and exhausting $(0,M)$ therefore yields, for $\Xi\ge0$,
\begin{equation}\label{eq:L2-Stieltjes-pairing}
 \ip{\xi}{Z}=-\int_{(0,M)}\Xi\dd D_aZ,\qquad Z\in\K.
\end{equation}
The Stieltjes integrals increase, and their limit is finite because $Z\xi\in L^1$.

If $\xi\in N_\K(X)$, the tests $X+c$ give $\Xi(M)=0$, and $0,2X$ give $\ip{\xi}{X}=0$. Testing next with $\one_{(a,M)}$ gives $-\Xi(a)\le0$. Formula \eqref{eq:L2-Stieltjes-pairing} proves the last condition in \eqref{eq:normalchar}. Conversely, that formula gives
$\ip{\xi}{Y-X}=-\int\Xi\dd D_aY\le0$ for every $Y\in\K$.

Continuity of $\Xi$ implies that it vanishes on $\supp D_aX$, whose complement consists of the positive fibres. Its values at their endpoints are zero, including $0,M$, and hence $\int_I\xi=0$ on each fibre. On the complement, $\xi=0$ almost everywhere, since an absolutely continuous function has zero derivative almost everywhere on its zero set. Thus $P_X\xi=0$ by \cref{lem:PX}.
\end{proof}

\begin{lemma}[Closure and the no-splitting property]\label{lem:compact-mono}
The cone $\K$ is strongly closed in $L^2$. If every positive fibre of $X\in\K$ is contained in a fibre of $Y\in\K$ (the no-splitting property), then $\sigma(Y)\subset\sigma(X)$ and $P_YP_X=P_Y$.
\end{lemma}
\begin{proof}
The distributional characterization
\[
 \K=\{X\in L^2:D_aX\ge0\}
 =\left\{X\in L^2:-\int X\varphi'\dd a\ge0
                    \text{ for }0\le\varphi\in C_c^\infty(0,M)\right\}
\]
proves closure.
Under the fibre-inclusion hypothesis, $Y\in\mathcal H_X$ by \cref{lem:PX}, so $\sigma(Y)\subset\sigma(X)$. The tower property yields $P_YP_X=P_Y$.
\end{proof}

\subsection[Quantiles and regularity estimates]{Quantiles and regularity estimates}\label{sec:compactness-traces}

\begin{lemma}[Monotone push-forward and quantile composition]\label{lem:quantile-details}
Let $\mu,\nu$ be finite nonnegative Radon measures on $\R$ with common mass $M>0$, and let $X_\mu,X_\nu$ denote their right-continuous quantiles. Let $T$ be a measurable map on a full $\mu$-measure set, finite and nondecreasing there, and suppose $\nu=T_\#\mu$. Then
\[
 X_\nu=T\circ X_\mu\qquad\text{a.e. on }(0,M).
\]
Only this full-measure composition is needed; no real-valued extension across an infinite endpoint trace is asserted.
\end{lemma}

The proof is given in \cref{app:measure}.

For arbitrary finite nonnegative Radon measures $\mu,\nu$ of common mass $M>0$, independently of any transport-map hypothesis,
\begin{equation}\label{eq:quantile-Winfty}
 W_\infty(\mu/M,\nu/M)
 =\|X_\mu-X_\nu\|_{L^\infty(0,M)}.
\end{equation}
Indeed, if a coupling $\pi\in\Pi(\mu/M,\nu/M)$ satisfies $|x-y|\le d<\infty$ almost everywhere, then
\[
 \mu(({-}\infty,x-d])\le\nu(({-}\infty,x])
 \le\mu(({-}\infty,x+d]).
\]
With the right-continuous inverse $X_\mu(a)=\inf\{x:\mu(({-}\infty,x])>a\}$, the strict superlevel sets give $X_\mu(a)-d\le X_\nu(a)\le X_\mu(a)+d$ for $0<a<M$. Taking the infimum over finite coupling bounds yields one inequality. The reverse follows from the coupling $(X_\mu,X_\nu)_\#(\dd a/M)$. If the quantile discrepancy is infinite, the first argument excludes every finite coupling bound, so the identity also holds in the extended sense. Neither bounded support, atomlessness nor attainment of the infimum is required.

\begin{lemma}[Canonical one-sided position traces]\label{lem:alltime-labels}
Suppose $X:[0,T]\to L^2(0,M)$ belongs to $\K$ at every time and satisfies
$\|X(t)-X(s)\|_\infty\le C|t-s|$. Its interior traces are finite, and there is a countable set $N\subset(0,M)$ such that
\[
 \begin{aligned}
 |X(a\pm,t)-X(a\pm,s)|&\le C|t-s|
 &&(0<a<M,\ s,t\in[0,T]),\\
 X(a-,t)&=X(a+,t)
 &&(a\in(0,M)\setminus N,\ t\in[0,T]).
 \end{aligned}
\]
An endpoint trace, $X(0+,t)$ or $X(M-,t)$, is either finite at all times, with the same estimate, or infinite with the same sign at all times. Infinite traces are not regarded as real-valued paths.
\end{lemma}

The proof is given in \cref{app:measure}.

For one-dimensional $BV$ representatives and one-sided traces, we use the conventions of \cite[Section 3.2]{AFP}.

\begin{lemma}[Interval averaging preserves the BV bound]\label{lem:BV-projection}
If $U\in BV(0,M)$ and $X\in\K$, then $P_XU$ has a regulated $BV$ representative and
\[
 \TV(P_XU)\le\TV(U).
\]
More generally, if $|U|\le R$ and $U|_J\in BV(J)$ on an open label interval $J\subset(0,M)$, allowing $0$ or $M$ as endpoints, then
\begin{equation}\label{eq:local-BV-projection}
 \TV_J(P_XU)\le\TV_J(U)+4R.
\end{equation}
\end{lemma}
\begin{proof}
Take ordered Lebesgue points of $P_XU$ and remove repetitions from each fibre. Associate with each remaining entry the normalized measure of its positive fibre, or a Dirac mass at its non-flat label. These probability measures $\lambda_0,\ldots,\lambda_k$ have ordered supports. Jensen's inequality gives
\[
 \sum_{j=1}^k\left|\int U\dd\lambda_j-\int U\dd\lambda_{j-1}\right|
 \le\int\cdots\int\sum_{j=1}^k|U(b_j)-U(b_{j-1})|
                  \dd\lambda_0\cdots\dd\lambda_k
 \le\TV(U).
\]
Thus all variation sums are bounded on a full-measure set. Taking one-sided limits gives the regulated $BV$ representative with this bound.

For the local assertion, extend $U|_J$ constantly by its endpoint traces, obtaining $\widetilde U$ with variation $\TV_J(U)$ and $|\widetilde U|\le R$. On $J$, $P_XU$ differs from $P_X\widetilde U$ only on the at most two fibres crossing the boundary of $J$. Each changes a constant on an initial or terminal interval by at most $2R$, contributing at most $2R$ to variation. If one fibre crosses both endpoints, the projection is constant on $J$. This proves \eqref{eq:local-BV-projection}, also for intervals adjoining $0$ or $M$.
\end{proof}

\begin{lemma}[OSL and local variation]\label{lem:OSL-BV}
Let $X\in\K$, $t>0$, and $B\in L^\infty(0,M)$ take values in $f([-R,R])$. The estimate
$B(b)-B(a)\le[X(b)-X(a)]/t$ for almost every $a<b$ is equivalent to the existence of a nonincreasing representative of $G=B-X/t$. On every open label interval $J$ of finite position oscillation,
\begin{equation}\label{eq:static-OSL-BV}
 \TV_J B\le2V_*+2\osc_JX/t,\qquad
 \TV_J(g(B))\le c_f^{-1}\TV_JB.
\end{equation}
The resulting representatives of $B$ and $g(B)$ have finite one-sided traces on $J$, including at $0$ or $M$ when $J$ adjoins that endpoint.
\end{lemma}
\begin{proof}
Under the pair inequality, $G^+(a)=\esssup_{b>a}G(b)$ is nonincreasing and finite at every interior label. Fubini gives $G^+(a)\le G(a)$ almost everywhere, while right averages at a Lebesgue point give the reverse inequality. The converse follows from monotonicity. On $J$,
$\TV_JG^+\le2V_*+\osc_JX/t$, so adding $\TV_JX/t$ gives the first estimate. The second follows from the $c_f^{-1}$-Lipschitz continuity of $g$. Finite variation gives both one-sided traces; no assertion is made at an endpoint of infinite position oscillation.
\end{proof}

\begin{lemma}[Tangency on the current fibres]\label{lem:fibre-tangent}
Let $Z\in W^{1,2}(0,T;L^2(0,M))$ have its continuous $L^2$ representative in $\K$ at every time. Then
\begin{equation}\label{eq:fibre-tangent}
 \dot Z(t)\in\mathcal H_{Z(t)}\qquad\text{for almost every }t\in(0,T).
\end{equation}
The no-splitting property is not assumed.
\end{lemma}
\begin{proof}
At each $L^2$-valued Lebesgue point $t$ of $\dot Z$, the two-sided difference quotients $[Z(t+\epsilon)-Z(t)]/\epsilon$ converge to $\dot Z(t)$ in $L^2$ by the Bochner integral identity.
On every positive fibre $I$ of $Z(t)$ the quotient is nondecreasing for $\epsilon>0$ and nonincreasing for $\epsilon<0$. Both cones are closed in $L^2(I)$, so their common limit is constant there. The exceptional time set was fixed before $I$; hence \cref{lem:PX} gives \eqref{eq:fibre-tangent} on all current fibres.
\end{proof}
\section[Bregman variational construction]{Bregman variational construction}\label{sec:existence}

Momentum conservation selects the Bregman velocity cost within the separable class characterized in \cref{prop:generator}. We first identify that cost and its collision barycenter, then construct a discrete evolution and identify its continuous limit. Variational time discretization using monotone transport maps and velocity projection was developed for classical Euler flow in \cite{CSWtime}. Here only time is discretized; all minimizations take place in $L^2(0,M)$. The Bregman cost determines the position increment, while conditional expectation updates the physical velocity and controls the discrepancy between displacement and transport velocities.

\subsection[\texorpdfstring{Averaging physical velocities and the Bregman cost}{Averaging physical velocities and the Bregman cost}]{Averaging physical velocities and the Bregman cost}\label{sec:collision-law}

Define $H(u)=\int_0^u f_R(r)\dd r$ and let $L=H^*$ be its Legendre transform, using the auxiliary function $f_R$ specified after \eqref{eq:f}. With $g=f_R^{-1}$, uniform convexity and maximization in one variable give
\begin{equation}\label{eq:dual}
 L(v)=vg(v)-H(g(v)),\quad L'(v)=g(v),\quad
 1/C_f\le L''(v)=1/f_R'(g(v))\le1/c_f.
\end{equation}
For a differentiable convex function $A$, we write
\[
 D_A(z\mid w)=A(z)-A(w)-A'(w)(z-w).
\]
The divergence is taken in Bregman's sense \cite{Bregman}. The first argument is the minimization variable, corresponding to the left-sided centroid convention in \cite[Theorem 3.2]{NN}.

\begin{proposition}[Bregman barycenter for cluster collisions]\label{prop:bary}
When masses $m_j>0$ with physical velocities $u_j\in[-R,R]$ merge conserving mass and momentum, their common physical velocity and transport velocity are
\begin{equation}\label{eq:bary}
 p=\frac{\sum_jm_ju_j}{\sum_jm_j},\qquad v=f(p)
 =\operatorname*{arg\,min}_{v\in\R}\sum_jm_jD_L(v\mid f(u_j)).
\end{equation}
\end{proposition}
\begin{proof}
Set $m=\sum_jm_j$. Conservation gives $mp=\sum_jm_ju_j$, and the velocity law gives the transport velocity $f(p)$. Denoting the functional in \eqref{eq:bary} by $\mathcal B$, we obtain
\[
 \mathcal B'(v)=\sum_jm_j\{g(v)-g(f(u_j))\}=m(g(v)-p),\qquad
 \mathcal B''(v)=mL''(v)\ge m/C_f.
\]
The functional is coercive and strictly convex. Its unique minimizer satisfies $g(v)=p$, as claimed.
\end{proof}

\begin{proposition}[Characterization within separable Bregman costs]\label{prop:generator}
Let $J=f([-R,R])$ and let $\Phi\in C^2(J)$ satisfy $\Phi''>0$. Suppose that, for every finite family $m_j>0$ and $u_j\in[-R,R]$, the objective
\[
 v\longmapsto\sum_jm_jD_\Phi(v\mid f(u_j)),\qquad v\in J,
\]
is minimized at $f(\sum_jm_ju_j/\sum_jm_j)$. Then
\begin{equation}\label{eq:generator}
 \Phi(v)=A L(v)+Bv+C\quad(v\in J),\qquad A>0.
\end{equation}
Conversely, every generator in \eqref{eq:generator} has this property, and $D_\Phi=A D_L$ on $J\times J$.
\end{proposition}
\begin{proof}
For two values $r<s$ and weights $1-\theta,\theta$ with $0<\theta<1$, the prescribed minimizer lies in the interior of $J$. Differentiation of the objective therefore gives
\[
 k((1-\theta)r+\theta s)=(1-\theta)k(r)+\theta k(s),
 \qquad k(u)=\Phi'(f(u)).
\]
Taking $r=-R$ and $s=R$ identifies $k$ on the whole interval with its affine interpolant, so $k(u)=Au+B$. Since $k'=\Phi''(f)f'>0$, we have $A>0$. Substitution of $u=g(v)$ gives $\Phi'(v)=Ag(v)+B$, and integration yields \eqref{eq:generator}. Conversely, affine terms cancel in a Bregman divergence and positive scaling preserves its minimizer. Thus $D_\Phi=A D_L$, and \cref{prop:bary} gives the asserted barycenter.
\end{proof}

The normalization $L'=g$ fixes the scale and linear part. By \eqref{eq:dual}, $L$ is quadratic on the transport-velocity interval exactly when $f$ is affine on the physical velocity interval.

\begin{proposition}[Obstruction to the classical free potential]\label{prop:classical-free}
Consider the Euclidean free projection
\begin{equation}\label{eq:classical-free-projection}
 Y_{\rm cl}(t)=\operatorname*{arg\,min}_{Y\in\K}
 \frac12\int_0^M|Y-X_0-tf(U_0)|^2\dd a,
\end{equation}
or the classical cumulative rule whose left and right mass labels at $x$ are the extreme minimizers of
\begin{equation}\label{eq:classical-free-potential}
 F_t^{\rm cl}(a;x)=\int_0^a\{X_0(b)+tf(U_0(b))-x\}\dd b.
\end{equation}
If either rule gives the transport velocity determined by the average physical velocity on every merged fibre for all bounded compatible data, then $f$ is affine on $[-R,R]$. The assertion is restricted to these two classical free-trajectory rules, both using the mass-label measure $\dd a$.
\end{proposition}
\begin{proof}
Fix $u>v$ in $[-R,R]$ and $0<\lambda<1$. Take $(X_0,U_0)=(0,u)$ on $(0,\lambda M)$ and $(X_0,U_0)=(1,v)$ on $(\lambda M,M)$. These data are compatible. The free positions $z_-=tf(u)$ and $z_+=1+tf(v)$ satisfy $z_->z_+$ for $t>t_c=(f(u)-f(v))^{-1}$. Averaging a monotone candidate on each initial block preserves monotonicity and decreases the quadratic cost. Thus \eqref{eq:classical-free-projection} reduces to
\[
 \min_{y_-\le y_+}\{\lambda(y_--z_-)^2+(1-\lambda)(y_+-z_+)^2\}.
\]
For $t>t_c$ the constraint is active, and its unique solution is
$y_-=y_+=q(t)=\lambda z_-+(1-\lambda)z_+$.
The potential \eqref{eq:classical-free-potential} has decreasing slopes $z_--x,z_+-x$, hence attains its minimum at $0$ or $M$. Since $F_t^{\rm cl}(M;x)=M(q(t)-x)$, the extreme minimizers select the same single fibre at $q(t)$.
Its physical velocity is $\lambda u+(1-\lambda)v$, so collision closure requires
\begin{equation}\label{eq:classical-affine-identity}
 f(\lambda u+(1-\lambda)v)=\lambda f(u)+(1-\lambda)f(v).
\end{equation}
Taking $u=R$, $v=-R$ and varying $\lambda$ proves affinity. Conversely, \eqref{eq:classical-affine-identity} holds for affine $f$.
\end{proof}

This obstruction does not exclude static functionals with additional label weights or nonlocal dependence.

\begin{lemma}[Dual Bregman estimates]\label{lem:dual-Bregman}
For the tangent-line extension $f_R$ fixed after \eqref{eq:f}, the functions $H$ and $L=H^*$ satisfy, for all real $r,s,v,w$,
\begin{align}
 D_L(v\mid f_R(r))&=L(v)+H(r)-rv,
 &D_H(r\mid s)&=D_L(f_R(s)\mid f_R(r)),\label{eq:Bregman-duality}\\
 \frac{c_f}{2}|r-s|^2&\le D_H(r\mid s)\le\frac{C_f}{2}|r-s|^2,
 &\frac{|v-w|^2}{2C_f}&\le D_L(v\mid w)\le\frac{|v-w|^2}{2c_f}.
 \label{eq:Bregman-quadratic}
\end{align}
In particular,
\begin{equation}\label{eq:Bregman-symmetric}
 D_L(v\mid w)+D_L(w\mid v)
 =(g(v)-g(w))(v-w)\ge C_f^{-1}|v-w|^2.
\end{equation}
\end{lemma}
\begin{proof}
By \eqref{eq:dual}, $L(f_R(r))=rf_R(r)-H(r)$ and $L'(f_R(r))=r$, which give \eqref{eq:Bregman-duality}. Taylor's formula with integral remainder, together with $c_f\le H''\le C_f$ and $C_f^{-1}\le L''\le c_f^{-1}$, gives \eqref{eq:Bregman-quadratic}. Adding the two $D_L$ terms and using $g'\ge C_f^{-1}$ proves \eqref{eq:Bregman-symmetric}.
\end{proof}

\subsection{Time discretization and estimates}\label{sec:time-increments}

For $h>0$, start from $X^0=X_0$, $U^0=U_0$ and define
\begin{equation}\label{eq:incremental-scheme}
 \begin{split}
 X^{n+1}&=\operatorname*{arg\,min}_{Y\in\K}
 \int_0^M D_L\!\left(\frac{Y-X^n}{h}\,\middle|\,f(U^n)\right)\dd a,\\
 U^{n+1}&=P_{X^{n+1}}U^n,\qquad
 V^{n+1}=\frac{X^{n+1}-X^n}{h}.
 \end{split}
\end{equation}
Since $L'(f(U^n))=U^n$ and $D_L(v\mid f(U^n))=L(v)+H(U^n)-U^nv$,
\begin{equation}\label{eq:incremental-Bregman}
 \begin{aligned}
 &\int_0^M\left[hL\!\left(\frac{Y-X^n}{h}\right)-U^nY\right]\dd a\\
 &\qquad=h\int_0^M D_L\!\left(\frac{Y-X^n}{h}\,\middle|\,f(U^n)\right)\dd a
 -\int_0^M[U^nX^n+hH(U^n)]\dd a.
 \end{aligned}
\end{equation}
Thus the two objectives have the same minimizer. We use the equivalent $hL-UY$ form on the left-hand side of \eqref{eq:incremental-Bregman} for the Euler--Lagrange calculation.

\begin{lemma}[Incremental minimization and no-splitting]\label{lem:variational-step}
For every $n\ge0$, \eqref{eq:incremental-scheme} has a unique minimizer. It satisfies
\begin{equation}\label{eq:incremental-bounds}
 X^n+hv_-\le X^{n+1}\le X^n+hv_+,
 \qquad |U^{n+1}|\le R.
\end{equation}
Every positive fibre of $X^n$ is contained in a positive fibre of $X^{n+1}$. Consequently,
\begin{equation}\label{eq:incremental-tower}
 \mathcal H_{X^{n+1}}\subset\mathcal H_{X^n},\qquad
 U^n=P_{X^n}U_0,\qquad U^m=P_{X^m}U^n\quad(m\ge n),
\end{equation}
and
\begin{equation}\label{eq:incremental-energy}
 \|U^n\|_2^2-\|U^m\|_2^2=\|U^n-U^m\|_2^2
 \qquad(m\ge n).
\end{equation}
Moreover,
\begin{equation}\label{eq:incremental-H-loss}
 \int_0^M[H(U^n)-H(U^{n+1})]\dd a
 =\int_0^M D_H(U^n\mid U^{n+1})\dd a
 \ge\frac{c_f}{2}\|U^n-U^{n+1}\|_2^2.
\end{equation}
The following normal-cone relations hold:
\begin{equation}\label{eq:incremental-normal}
 \xi^{n+1}:=U^n-g(V^{n+1})\in N_\K(X^{n+1}),
 \qquad U_0-U^{n+1}\in N_\K(X^{n+1}).
\end{equation}
\end{lemma}
\begin{proof}
Fix $X\in\K\subset L^2(0,M)$ and $U=P_XU$, $|U|\le R$. By \eqref{eq:incremental-Bregman}, \cref{lem:dual-Bregman} and the Lipschitz continuity of $g=L'$, the $hL-UY$ objective is finite, continuous, coercive and $1/(C_fh)$-strongly convex on $L^2$. The direct method on the closed convex cone $\K$ therefore gives a unique minimizer. Denote the minimizer by $Y$ and set $p=P_YU$.

The scalar integrand has its unique minimum at $y=X(a)+hf(U(a))$, between the nondecreasing barriers $X+hv_-$ and $X+hv_+$. Clipping a candidate between these barriers preserves monotonicity and strictly lowers the objective if a barrier is violated on a set of positive measure. This gives the position bounds in \eqref{eq:incremental-bounds}; conditional averaging gives the physical velocity bound.

On an old positive fibre $I$, both $X$ and $U$ are constant. Replacing $Y|_I$ by its average preserves monotonicity and, by strict Jensen, decreases the objective unless $Y$ is constant on $I$. This proves no-splitting. Starting from $U_0=P_{X_0}U_0$, the tower property in \cref{lem:compact-mono} and orthogonality give \eqref{eq:incremental-tower}--\eqref{eq:incremental-energy}. Since $\langle U^n-U^{n+1},f(U^{n+1})\rangle=0$, \cref{lem:dual-Bregman} gives \eqref{eq:incremental-H-loss}. Conditional Jensen also makes $\int\eta(U^n)$ nonincreasing for every continuous convex $\eta$ on $[-R,R]$.

The $L^2$ gradient is $g((Y-X)/h)-U$, so the variational inequality is
\[
 \langle g((Y-X)/h)-U,Z-Y\rangle\ge0
 \qquad(Z\in\K),
\]
which is the first relation in \eqref{eq:incremental-normal}.

For the second relation, put $Q=g((Y-X)/h)$ and $\xi=U-Q$. By \cref{lem:normal}, $P_Y\xi=0$, so $P_YQ=p$. On a positive fibre $I=(l,r)$ of $Y$, the function $Q$ is nonincreasing and has average $p_I$. Hence
\[
 \int_l^a(Q-p_I)\dd b\ge0\quad(l<a<r),
 \qquad\int_I(Q-p_I)\dd b=0.
\]
On the non-flat set, $Q=U=p$ almost everywhere. Thus the primitive of $Q-p$ is nonnegative, vanishes off the positive fibres and at their endpoints, and pairs to zero with $D_aY$. By \cref{lem:normal}, $Q-p\in N_\K(Y)$, and therefore $U-p\in N_\K(Y)$.

Finally, if $\zeta\in N_\K(X)$ and $Y\in\mathcal H_X\cap\K$, then $\langle\zeta,Z\rangle\le0$ for $Z\in\K$ and $\langle\zeta,Y\rangle=0$ by \cref{lem:normal}; hence $\zeta\in N_\K(Y)$. No-splitting and addition of normals, starting from $U_0-U^0=0$, prove the second assertion of \eqref{eq:incremental-normal}.
\end{proof}

\begin{samepage}
The normal-cone relation controls the discrepancy between displacement and transport velocities by the change in physical velocity under conditional averaging.

\begin{lemma}[Displacement and transport velocities]\label{lem:incremental-velocity-error}
For every $N\ge1$,
\begin{equation}\label{eq:velocity-error-sum}
 \sum_{n=0}^{N-1}h\|V^{n+1}-f(U^{n+1})\|_2^2
 \le C_f^2h\bigl(\|U_0\|_2^2-\|U^N\|_2^2\bigr).
\end{equation}
\end{lemma}
\end{samepage}
\begin{proof}
Write $V=V^{n+1}$, $p=U^{n+1}$ and $\xi=\xi^{n+1}$. Testing the normal-cone relation with $X^n$ gives $\langle\xi,V\rangle\ge0$. Since $f(p)\in\mathcal H_{X^{n+1}}$, \cref{lem:normal} gives $\langle\xi,f(p)\rangle=0$. By the symmetric Bregman estimate,
\begin{align*}
 C_f^{-1}\|V-f(p)\|_2^2
 &\le\int_0^M\{D_L(V\mid f(p))+D_L(f(p)\mid V)\}\dd a\\
 &=\langle g(V)-p,V-f(p)\rangle\\
 &=\langle U^n-p,V-f(p)\rangle-\langle\xi,V\rangle\\
 &\le\|U^n-p\|_2\|V-f(p)\|_2.
\end{align*}
Squaring $\|V-f(p)\|_2\le C_f\|U^n-p\|_2$, multiplying by $h$ and summing \eqref{eq:incremental-energy} proves the claim.
\end{proof}

For an interval $I=(\alpha,\beta)$, a prefix is a subinterval adjoining $\alpha$, and a suffix is a subinterval adjoining $\beta$; the whole interval $I$ is allowed in either case. Endpoints do not affect integrals, and averages are used only on intervals of positive length.

\begin{lemma}[Discrete Oleinik estimate]\label{lem:discrete-OSL}
For every $n\ge1$ and almost every ordered pair $a<b$,
\begin{equation}\label{eq:discrete-OSL}
 f(U^n(b))-f(U^n(a))\le\frac{X^n(b)-X^n(a)}{nh}.
\end{equation}
The functions $f(U^n)-X^n/(nh)$ have nonincreasing representatives, and $U^n\in BV_{\mathrm{loc}}(0,M)$ for $n\ge1$. On any label interval $J$ on which $X^n$ has finite oscillation,
\begin{equation}\label{eq:discrete-local-BV}
 \TV_J(f(U^n))\le2V_*+\frac{2\operatorname{osc}_JX^n}{nh},
 \qquad \TV_J(U^n)\le c_f^{-1}\TV_J(f(U^n)).
\end{equation}
\end{lemma}
\begin{proof}
We first derive endpoint bounds on each new fibre, establish the first-step estimate without physical velocity traces, and then propagate the bound by induction.

Fix a step and abbreviate
\[
 X=X^n,\quad U=U^n,\quad Y=X^{n+1},\quad p=U^{n+1},
 \quad Q=g((Y-X)/h),\quad\Xi(a)=\int_0^a(U-Q)\dd b.
\]
By \cref{lem:normal,lem:variational-step}, $\Xi\ge0$, and $\Xi$ vanishes at the endpoints of every positive fibre of $Y$.

\emph{Endpoint bounds.} Let $I=(l,r)$ be such a fibre, at position $y$, with physical velocity $p_I$. The displacement barriers give
\[
 y-hv_+\le X(a)\le y-hv_-\qquad\hbox{for a.e. }a\in I.
\]
Thus the interior endpoint traces of $X|_I$ are finite, including when $l=0$ or $r=M$. The corresponding traces of $Q=g((y-X)/h)$ are finite as well. On $I$,
\begin{equation}\label{eq:OSL-prefix-suffix}
 \int_I Q\dd a=|I|p_I,\qquad
 \int_l^c(U-Q)\dd a\ge0,\qquad
 \int_c^r(U-Q)\dd a\le0\quad(l<c<r).
\end{equation}
Since $X$ is nondecreasing, $Q=g((y-X)/h)$ is nonincreasing on $I$. Its average and its endpoint traces give
\begin{equation}\label{eq:OSL-displacement-ends}
 \frac{y-X(r-)}h\le f(p_I)\le\frac{y-X(l+)}h.
\end{equation}
This argument uses no physical velocity traces and applies also to $n=0$.

If $U|_I\in BV(I)$, divide the prefix and suffix inequalities in \eqref{eq:OSL-prefix-suffix} by their interval lengths and let the lengths decrease to zero. The interior one-sided limits exist, and
\[
 U(l+)\ge Q(l+)\ge p_I,
 \qquad U(r-)\le Q(r-)\le p_I.
\]
Thus
\begin{equation}\label{eq:OSL-state-ends}
 U(r-)\le p_I\le U(l+).
\end{equation}
Only traces from within $I$ occur here; their existence will follow from the preceding step's discrete Oleinik estimate and the displacement barriers, also when $l=0$ or $r=M$.

On the complement of the positive fibres, remove their countably many endpoints and the null sets in \cref{lem:PX,lem:normal}. At every remaining Lebesgue point $a$,
\begin{equation}\label{eq:OSL-nonflat}
 p(a)=U(a)=Q(a),\qquad
 f(p(a))=\frac{Y(a)-X(a)}h.
\end{equation}
Such a label is treated below as a singleton fibre: both old position values are $X(a)$, and, whenever physical velocity traces are used, both old physical velocity values are $U(a)$. Hence \eqref{eq:OSL-displacement-ends} and \eqref{eq:OSL-state-ends} become equalities. Here the labels are chosen in the common full-measure set of Lebesgue points of all the functions involved.

\emph{The first time step.} Consider two distinct ordered new fibres, allowing these singletons, with positions $y_A<y_B$, physical velocities $p_A,p_B$ and transport velocities $v_A=f(p_A)$, $v_B=f(p_B)$. Write $x_A^+$ for the old position at the right interior endpoint of the left fibre, and $x_B^-$ for the old position at the left interior endpoint of the right fibre; for a singleton use its point value. Set
\[
 D=y_B-y_A,\qquad d=x_B^--x_A^+\ge0,
 \qquad\Delta v=v_B-v_A.
\]
The lower bound for the left fibre and the upper bound for the right fibre in \eqref{eq:OSL-displacement-ends} give
\begin{equation}\label{eq:OSL-new-gap}
 h\Delta v\le D-d.
\end{equation}
For $n=0$, this proves the desired bound at time $h$. Two labels in the same new fibre have both position and transport-velocity difference zero. Together with \eqref{eq:OSL-nonflat}, this proves \eqref{eq:discrete-OSL} at the first step for almost every ordered pair.

\emph{Induction.} Suppose \eqref{eq:discrete-OSL} holds at step $n\ge1$. By \cref{lem:OSL-BV}, $f(U)-X/(nh)$ has a nonincreasing representative and \eqref{eq:discrete-local-BV} holds at the old step. The displacement barriers bound $X|_I$ for every new positive fibre, including those touching $0$ or $M$. Hence $U|_I\in BV(I)$, and the traces in \eqref{eq:OSL-state-ends} exist. The discrete Oleinik inequality at the old step extends to these traces by approaching the endpoints from within the respective fibres through a common full-measure set; at a shared cut the approaches are from opposite sides. Non-flat labels use the Lebesgue values in \eqref{eq:OSL-nonflat}. All position traces used here are finite.

Let $u_A^+$ and $u_B^-$ denote the corresponding old physical velocity traces, or the singleton values in \eqref{eq:OSL-nonflat}. By \eqref{eq:OSL-state-ends} and the induction hypothesis,
\begin{equation}\label{eq:OSL-old-gap}
 \Delta v\le f(u_B^-)-f(u_A^+)\le\frac d{nh}.
\end{equation}
Adding \eqref{eq:OSL-new-gap} and $nh\Delta v\le d$ gives
\[
 (n+1)h\Delta v\le D.
\]
This proves the induction for positive fibres and non-flat labels. Apply \cref{lem:OSL-BV} at the new time to obtain \eqref{eq:discrete-local-BV}. The first step required no $BV$ assumption on $U_0$.
\end{proof}

\subsection[Compactness and identification of the limit]{Compactness and identification of the limit}\label{sec:approximation-compactness}

For $nh<t\le(n+1)h$, define
\begin{equation}\label{eq:time-interpolants}
 \begin{aligned}
 X_h(t)&=X^n+(t-nh)V^{n+1},&\quad \overline X_h(t)&=X^{n+1},\\
 U_h(t)&=U^{n+1},& V_h(t)&=V^{n+1}.
 \end{aligned}
\end{equation}
At zero set $X_h(0)=\overline X_h(0)=X_0$ and $U_h(0)=U_0$. Iteration of the displacement barriers gives, on $[0,T]$,
\begin{equation}\label{eq:interpolant-bounds}
 \begin{gathered}
 |X_h(t)-X_0|\le V_*t,\qquad
 |\overline X_h(t)-X_0|\le V_*(t+h),\qquad
 |U_h|\le R,\quad |V_h|\le V_*,\\
 \|X_h(t)\|_2\le\|X_0\|_2+\sqrt M V_*t,\qquad
 \|\overline X_h-X_h\|_{L^\infty((0,T)\times(0,M))}\le hV_*.
 \end{gathered}
\end{equation}
The following compactness result uses nested conditional averaging and does not depend on the position range. The $BV$ approximation in its proof is used only to estimate label translations; no $BV$ assumption is imposed on the initial physical velocity.

\begin{lemma}[Strong compactness of nested conditional averages]\label{lem:nested-state-compactness}
Let $A_k\to A$ in $L^2(0,M)$, with $|A_k|,|A|\le R$. Suppose $Z_k(t)\in\K\subset L^2(0,M)$ are jointly Borel maps with the no-splitting property in time, and set
\[
 W_k(t)=P_{Z_k(t)}A_k.
\]
Then $(W_k)$ is relatively compact in $L^2((0,T)\times(0,M))$. No position bound is assumed. In particular, the physical velocities $U_h$ in \eqref{eq:time-interpolants} have this compactness as $h\downarrow0$.
\end{lemma}
\begin{proof}
The conditional averages can be chosen jointly measurable without any time regularity of $Z_k$. Indeed, the denominator
\[
 m_k(a,t)=\int_0^M\one_{\{Z_k(b,t)=Z_k(a,t)\}}\dd b
\]
is measurable, as is the same integral with the factor $A_k(b)$. Where $m_k>0$, take their ratio; where $m_k=0$, take $A_k(a)$. By \cref{lem:PX}, this is a representative of $W_k(t)$ for every time. The same construction applies to every auxiliary conditional average used below.

For $s<t$, the tower property and orthogonality give
\begin{equation}\label{eq:nested-translation-energy}
 \|W_k(t)-W_k(s)\|_2^2=e_k(s)-e_k(t),
 \qquad e_k(t)=\|W_k(t)\|_2^2\le MR^2.
\end{equation}
Thus, for $0<\tau<T$,
\begin{align}
 \int_0^{T-\tau}\|W_k(t+\tau)-W_k(t)\|_2^2\dd t
 &=\int_0^\tau e_k(t)\dd t-\int_{T-\tau}^T e_k(t)\dd t\notag\\
 &\le MR^2\tau.
 \label{eq:state-time-translates}
\end{align}

By smooth approximation and truncation, choose $A^j\in BV(0,M)$ with $|A^j|\le R$ and $\epsilon_j=\|A^j-A\|_2\to0$. Put $W_k^j=P_{Z_k}A^j$ and $\delta_k=\|A_k-A\|_2$. Projection contraction and \cref{lem:BV-projection} give, at every time,
\[
 \|W_k-W_k^j\|_2\le\delta_k+\epsilon_j,
 \qquad |W_k^j|\le R,
 \qquad\TV(W_k^j)\le\TV(A^j).
\]
For $0<\delta<M$, a $BV$ function $z$ satisfies
\[
 \int_0^{M-\delta}|z(a+\delta)-z(a)|\dd a
 \le\int_0^{M-\delta}|Dz|((a,a+\delta])\dd a
 \le\delta\TV(z),
\]
where the last inequality follows by Fubini. Since $|z|\le R$ implies $|z(a+\delta)-z(a)|^2\le2R|z(a+\delta)-z(a)|$, decomposition through $W_k^j$ yields
\begin{equation}\label{eq:state-space-translates}
 \int_0^{M-\delta}|W_k(a+\delta,t)-W_k(a,t)|^2\dd a
 \le6(\delta_k+\epsilon_j)^2+6R\delta\TV(A^j).
\end{equation}

Extend $W_k$ by zero outside $(0,T)\times(0,M)$, and denote this extension by $\widetilde W_k$. The two time boundary strips add at most $2MR^2\tau$ to \eqref{eq:state-time-translates}. The two label boundary strips add at most $2TR^2\delta$ to the time integral of \eqref{eq:state-space-translates}. Consequently,
\begin{align}
 \|\widetilde W_k(\cdot+\tau,\cdot)-\widetilde W_k\|_{L^2(\R^2)}^2
 &\le3MR^2\tau,\label{eq:state-zero-time}\\
 \|\widetilde W_k(\cdot,\cdot+\delta)-\widetilde W_k\|_{L^2(\R^2)}^2
 &\le6T(\delta_k+\epsilon_j)^2
 +6RT\delta\TV(A^j)+2TR^2\delta.
 \label{eq:state-zero-space}
\end{align}
A change of variables and the triangle inequality give the corresponding bounds for negative and general two-dimensional translations. First choose $j$ to make $\epsilon_j$ small, then a sequence tail on which $\delta_k$ is small, and finally let the translation sizes tend to zero. The finite omitted family has a common vanishing modulus by $L^2$ translation continuity. Thus
\begin{equation}\label{eq:state-FK-modulus}
 \lim_{r\downarrow0}\sup_k\sup_{|z|_\infty\le r}
 \|\widetilde W_k(\cdot+z)-\widetilde W_k\|_{L^2(\R^2)}=0.
\end{equation}

The bounded zero extensions have common compact support and satisfy \eqref{eq:state-FK-modulus}; the Fr\'echet--Kolmogorov criterion \cite[Corollary 4.27]{Brezis} gives strong $L^2$ compactness. For the scheme, use $A_k=A=U_0$ and $Z_k=\overline X_{h_k}$, with no-splitting from \cref{lem:variational-step}. The approximants $A^j$ enter only the estimates, whose constants are independent of the position range.
\end{proof}

\begin{proposition}[The continuous Lagrangian limit]\label{prop:variational-limit}
Along a subsequence $h\downarrow0$, for every fixed $T<\infty$,
\begin{equation}\label{eq:variational-strong-limit}
 U_h\to p,\quad f(U_h)\to f(p),\quad V_h\to f(p)
 \quad\text{in }L^2((0,T)\times(0,M)),
 \qquad X_h\to X\quad\text{in }C([0,T];L^2).
\end{equation}
The limit satisfies $X(0)=X_0$, $X(t)\in\K$, and
\begin{align}
 X&\in C([0,T];L^2)\cap W^{1,\infty}(0,T;L^2),\label{eq:main-reg}\\
 \|X(t)-X(s)\|_\infty&\le V_*|t-s|\qquad(0\le s,t\le T).\label{eq:timeLip}
\end{align}
With the canonical Lagrangian physical velocity $p(t)=P_{X(t)}U_0$ defined at every time, $X_t=f(p)$ in $L^2$ for almost every time, and
\begin{equation}\label{eq:limit-normal}
 p(t)=P_{X(t)}U_0,\qquad U_0-p(t)\in N_\K(X(t))
 \quad\text{for almost every }t.
\end{equation}
The map $X$ has the no-splitting property from time zero, and $|X(t)-X_0|\le V_*t$ almost everywhere in the label. For the canonical monotone traces at interior labels,
\begin{equation}\label{eq:limit-gap}
 0\le X(b,t)-X(a,t)
 \le\frac ts\,[X(b,s)-X(a,s)]
 \qquad(a<b,\ 0<s<t).
\end{equation}
In particular, for almost every positive time and almost every ordered label pair,
\begin{equation}\label{eq:limit-material-OSL}
 f(p(b,t))-f(p(a,t))\le\frac{X(b,t)-X(a,t)}t.
\end{equation}
\end{proposition}
\begin{proof}
The proof proceeds in three steps.

\emph{Step 1. Strong convergence and limiting motion.} By \cref{lem:nested-state-compactness}, extract $U_h\to U$ strongly in spacetime $L^2$. Since $|U_h|\le R$ and $f$ is Lipschitz on $[-R,R]$, $|U|\le R$ and $f(U_h)\to f(U)$ strongly. By \cref{lem:incremental-velocity-error}, with $N=\lceil T/h\rceil$,
\[
 \|V_h-f(U_h)\|_{L^2((0,T)\times(0,M))}^2
 \le C_f^2h\|U_0\|_2^2\longrightarrow0.
\]
Thus $V_h\to f(U)$ strongly. Define
\begin{equation}\label{eq:limit-integral-first}
 X(t)=X_0+\int_0^t f(U(r))\dd r.
\end{equation}
Since $\dot X_h=V_h$,
\[
 \sup_{0\le t\le T}\|X_h(t)-X(t)\|_2
 \le\sqrt T\|V_h-f(U)\|_{L^2((0,T)\times(0,M))}\longrightarrow0.
\]
By \eqref{eq:interpolant-bounds}, the piecewise constant interpolants $\overline X_h$ have the same limit, and the time-Lipschitz, displacement and $L^2$ bounds pass to $X\in C([0,T];\K)$, with $X_t=f(U)$. In particular, $|X(t)|>K>2V_*T$ implies $|X_0|>K/2$ and $|X(t)|\le2|X_0|$, so
\begin{equation}\label{eq:constructed-moment-tail}
 \sup_{0\le t\le T}\int_{\{|X(t)|>K\}}|X(t)|^2\dd a
 \le4\int_{\{|X_0|>K/2\}}|X_0|^2\dd a
 \longrightarrow0\qquad(K\to\infty).
\end{equation}

\emph{Step 2. Identification of the Lagrangian physical velocity and normal-cone relation.} For bounded Lipschitz $\phi$ and bounded time tests $\eta$, conditional expectation gives $\int_0^T\eta\langle U_h-U_0,\phi(\overline X_h)\rangle\dd t=0$. Strong convergence passes this identity to
\begin{equation}\label{eq:limit-conditional-tests}
 \int_0^T\eta(t)\int_0^M(U-U_0)\phi(X)\dd a\dd t=0.
\end{equation}
Here $\phi(\overline X_h)\to\phi(X)$ strongly in $L^2$ and $|U_h-U_0|\le2R$. Choose a countable uniformly dense family of compactly supported Lipschitz functions in $C_0(\R)$. Outside one null time set, \eqref{eq:limit-conditional-tests} gives zero for each corresponding spatial integral, hence for every $\phi\in C_0$ by the bound $2RM\|\phi\|_\infty$. Therefore $X(t)_\#((U(t)-U_0)\dd a)=0$, which is equivalent to
\begin{equation}\label{eq:limit-projected-state}
 P_{X(t)}U(t)=P_{X(t)}U_0\qquad\text{for almost every }t.
\end{equation}
To identify $U$ itself, apply \cref{lem:fibre-tangent}, which yields $X_t(t)\in\mathcal H_{X(t)}$ almost everywhere. Since $U=g(X_t)$, we also have $U(t)\in\mathcal H_{X(t)}$; thus \eqref{eq:limit-projected-state} identifies $U=P_XU_0$. Set $p(t)=P_{X(t)}U_0$ at every time.

For $Z\in\K\cap L^\infty$ and nonnegative bounded $\eta$, \eqref{eq:incremental-normal} gives
\[
 \int_0^T\eta(t)\langle U_0-U_h(t),Z-\overline X_h(t)\rangle\dd t\le0.
\]
Both factors converge strongly in $L^2$. Passing to the limit first on a countable $L^2$-dense family of bounded elements of $\K$, then on all $Z\in\K$ by clipping and continuity, proves \eqref{eq:limit-normal}. No continuity of the projection operators is assumed.

\emph{Step 3. Oleinik condition and no-splitting.} Extract pointwise almost-everywhere convergence in spacetime. With $t_h=h\lceil t/h\rceil$, \eqref{eq:discrete-OSL} reads
\[
 t_h[f(U_h(b,t))-f(U_h(a,t))]
 \le\overline X_h(b,t)-\overline X_h(a,t).
\]
For almost every $t$ it passes to the limit on a full-measure set of labels, proving \eqref{eq:limit-material-OSL}.

The Bochner identity and Fubini yield absolutely continuous paths on one full-measure label set. Agreement with monotone representatives at rational times and time continuity preserve order there for all times. For almost every ordered label pair, $d(t)=X(b,t)-X(a,t)$ is nonnegative and absolutely continuous, and
\[
 (d(t)/t)'=\frac{td'(t)-d(t)}{t^2}\le0\quad\text{for a.e. }t>0.
\]
Integration gives \eqref{eq:limit-gap} for all positive time pairs. Approaching any two cuts from the prescribed sides through the full-measure pair set extends the inequality to the monotone one-sided traces. Interior traces are finite; endpoint traces are used only when finite, a property independent of time by finite displacement. Their time-Lipschitz representatives are supplied by \cref{lem:alltime-labels}. Hence every gap vanishing at a positive time remains zero.

At time zero, \cref{lem:variational-step} and affine interpolation give $X_h(t)\in\mathcal H_{X_0}$ for every $t$. The strong limit belongs to this closed subspace, so initial positive fibres do not split either.
\end{proof}

\subsection{Eulerian reconstruction and global existence}\label{sec:euler-reconstruction}

With the canonical Lagrangian physical velocity $p(t)=P_{X(t)}U_0$ at every time, set
\begin{equation}\label{eq:euler-push}
 \rho_t=X(t)_\#\dd a,\qquad
 \qmeas_t=X(t)_\#(U_0\dd a)=X(t)_\#(p(t)\dd a).
\end{equation}
The last identity is conditional expectation tested against bounded functions of $X(t)$; in particular $|\qmeas_t|\le R\rho_t$.

\begin{lemma}[Jointly Borel representatives]\label{lem:joint}
The maps $X,p$ have jointly Borel representatives. The measure kernels $\rho_t,\qmeas_t$ are narrowly continuous. There is a jointly Borel $u^\sharp(x,t)\in[-R,R]$ such that, at every time, $\qmeas_t=u^\sharp(\cdot,t)\rho_t$ and
\begin{equation}\label{eq:pullback}
 u^\sharp(X(a,t),t)=p(a,t)\qquad\text{for almost every }a.
\end{equation}
\end{lemma}

The proof in \cref{app:measure} applies by \cref{prop:variational-limit}, \eqref{eq:constructed-moment-tail} and \eqref{eq:euler-push}.

\begin{proposition}[Eulerian fluxes and weak equations]\label{prop:closure}
The measure pair in \eqref{eq:euler-push}, with physical velocity $u^\sharp$, satisfies both weak conservation laws \eqref{eq:weak-m}--\eqref{eq:weak-q}. At every time,
\begin{equation}\label{eq:fluxclosure}
 X(t)_\#(f(p(t))\dd a)=f(u^\sharp)\rho_t,
 \qquad
 X(t)_\#(p(t)f(p(t))\dd a)=u^\sharp f(u^\sharp)\rho_t.
\end{equation}
\end{proposition}
\begin{proof}
Both identities in \eqref{eq:fluxclosure} follow by testing with bounded Borel functions and using \eqref{eq:pullback}. On the full-measure set of absolutely continuous label paths, the chain rule for $\phi\in C_c^1(\R\times[0,T))$ gives, after integration,
\begin{equation}\label{eq:material-mass-test}
 \int_0^T\!\int_0^M
 [\phi_t(X,t)+f(p)\phi_x(X,t)]\dd a\dd t
 +\int_0^M\phi(X_0,0)\dd a=0.
\end{equation}
All integrands are bounded on the finite mass--time rectangle; push-forward yields \eqref{eq:weak-m}. Multiplying the pathwise identity by $U_0(a)$ gives the second law: at almost every time, both $\psi_t(X,t)$ and $f(p)\psi_x(X,t)$ belong to $\mathcal H_{X(t)}$, so conditional expectation replaces $U_0$ by $p$ in these terms. The initial term remains $\int U_0\psi(X_0,0)\dd a$. Equations \eqref{eq:data}, \eqref{eq:pullback} and \eqref{eq:fluxclosure} then give \eqref{eq:weak-q}, without differentiating $p$ in time.
\end{proof}

\begin{lemma}[Initial physical velocity and energy traces]\label{lem:constructed-energy}
The canonical Lagrangian physical velocity satisfies
\begin{equation}\label{eq:initial-strong-state}
 \begin{gathered}
 \lim_{t\downarrow0}\|p(t)-U_0\|_2=0,\\
 \lim_{t\downarrow0}\int\varphi(x)|u^\sharp(x,t)|^2\dd\rho_t
 =\int\varphi(x)|u_0|^2\dd\rho_0
 \qquad(\varphi\in C_b(\R)).
 \end{gathered}
\end{equation}
\end{lemma}
\begin{proof}
Since $U_0\in\mathcal H_{X_0}$, choose $\zeta_j\in C_c^\infty(\R)$ with $\varepsilon_j=\|\zeta_j(X_0)-U_0\|_2\to0$. By \cref{lem:projection-approximation} and the displacement bound,
\begin{equation}\label{eq:initial-trace-bound}
 \|p(t)-U_0\|_2
 \le\varepsilon_j+\operatorname{Lip}(\zeta_j)\|X(t)-X_0\|_2
 \le\varepsilon_j+\sqrt M\operatorname{Lip}(\zeta_j)V_*t.
\end{equation}
Let $t\downarrow0$, then $j\to\infty$. For every $\varphi\in C_b(\R)$, \eqref{eq:pullback} gives
\[
 \begin{aligned}
 &\left|\int\varphi(x)|u^\sharp(x,t)|^2\dd\rho_t
       -\int\varphi(x)|u_0(x)|^2\dd\rho_0\right|\\
 &\quad\le 2R\sqrt M\|\varphi\|_\infty\|p(t)-U_0\|_2
 +\int_0^M|U_0|^2|\varphi(X(t))-\varphi(X_0)|\dd a
 \longrightarrow0.
 \end{aligned}
\]
The second term tends to zero by finite displacement and dominated convergence. Thus the energy-measure trace holds through all positive times, and in particular on the fixed full-measure set required in \cref{def:entropy}.
\end{proof}

\begin{lemma}[Jointly Borel monotone-envelope extension]\label{lem:OSL-extension}
Let $X(a,t),B(a,t)$ be jointly Borel on $(0,M)\times(0,T]$, with $X(t)\in\K\subset L^2(0,M)$ and $v_-\le B\le v_+$. Suppose that, for almost every time, $B(t)\in\mathcal H_{X(t)}$ and
\begin{equation}\label{eq:labelOSL-pre}
 B(b,t)-B(a,t)\le\frac{X(b,t)-X(a,t)}t
 \quad\text{for almost every }a<b.
\end{equation}
There is a jointly Borel $b(x,t)\in[v_-,v_+]$ satisfying the Oleinik condition \eqref{eq:OSL} and $b(X(a,t),t)=B(a,t)$ for almost every $(a,t)$.
\end{lemma}
\begin{proof}
Set $G(a,t)=B(a,t)-X(a,t)/t$, and take the essential supremum over an empty or null set to be $-\infty$. For every $x\in\R$ and $t>0$, define
\begin{align}
 S(x,t)&=\esssup_{a:\,X(a,t)\ge x}G(a,t),\label{eq:S-envelope}\\
 H_0(x,t)&=\max\{v_--x/t,S(x,t)\},\label{eq:H-envelope}\\
 b(x,t)&=x/t+H_0(x,t).
 \label{eq:b-envelope}
\end{align}
The joint measurability follows from the exact superlevel-set identity
\[
 \{S>r\}
 =\left\{(x,t):\int_0^M
 \one_{\{X(a,t)\ge x\}}\one_{\{G(a,t)>r\}}\dd a>0\right\}
 \qquad(r\in\mathbb Q).
\]
On the selector set $X(a,t)\ge x$ one has $G(a,t)\le v_+-x/t$, and therefore $S(x,t)\le v_+-x/t$, including when the selector is null. Thus
\[
 v_--x/t\le H_0(x,t)\le v_+-x/t,
 \qquad v_-\le b(x,t)\le v_+.
\]
Thus $b$ is everywhere finite. The selector sets decrease with $x$, so $S(\cdot,t)$ and $H_0(\cdot,t)$ are nonincreasing at every time. Hence $b=x/t+H_0$ satisfies \eqref{eq:OSL}.

Fix a time at which \eqref{eq:labelOSL-pre} and fibre constancy hold, so $G$ is essentially nonincreasing. On a positive fibre $I$ at position $x$, the non-strict selector $X\ge x$ includes all of $I$; hence $S(x,t)=G|_I$. At a Lebesgue point $a$ of $G$ in the non-flat set, excluding fibre endpoints, $\{a':X(a',t)\ge X(a,t)\}$ is the right label tail modulo null sets. Its essential supremum equals the right limit $G(a,t)$. Thus in both cases
\[
 S(X(a,t),t)=G(a,t)\ge v_--X(a,t)/t,\qquad
 b(X(a,t),t)=B(a,t)
\]
for almost every label. Fubini gives the asserted spacetime equality.
\end{proof}

\begin{proof}[Proof of the existence assertion in \textup{\cref{thm:main}}]
Apply \cref{lem:nested-state-compactness,prop:variational-limit} successively on $[0,1],[0,2],\ldots$ to nested subsequences of $h_k\downarrow0$. The diagonal limits agree on overlaps, independently of uniqueness, and define a global map $X$ with the stated path equation, regularity and no-splitting property. The push-forwards \eqref{eq:euler-push} attain the prescribed initial data by \eqref{eq:data}. Their narrow continuity follows from \cref{lem:joint}, the two weak conservation laws from \cref{prop:closure}, and initial weak continuity of the kinetic energy measure from \cref{lem:constructed-energy}.

Apply \cref{lem:OSL-extension} to $B=f(p)$, using \eqref{eq:limit-material-OSL}. Its formula defines one jointly Borel coefficient for all $t>0$, without a support bound. By \eqref{eq:pullback}, this coefficient equals $f(u^\sharp)$ for $\rho_t\dd t$-almost every point. Hence \cref{def:entropy} holds on every finite interval. This proves global existence and the representation \eqref{eq:main-state}--\eqref{eq:main-push} for the constructed solution, with the regularity and displacement bounds of \cref{prop:variational-limit}. Its $L^2$ position bound and the push-forward identity give \eqref{eq:main-moment}.
\end{proof}

\section[Lagrangian structure of entropy solutions]{Lagrangian structure of entropy solutions}\label{sec:lifting}

Throughout this section, $(\rho,u)$ is an arbitrary entropy solution in the sense of \cref{def:entropy}. We recover its Lagrangian representation and establish compression inequalities, then derive one-sided physical velocity information and backward-fan geometry.

\subsection{Lagrangian representation}\label{sec:common-transport}

\begin{lemma}[Finite displacement, moment propagation and the momentum measure]\label{lem:canonical}
Let $X(t)$ be the quantile of the mass measure of an arbitrary entropy solution. Then, for all $0\le s<t\le T$,
\begin{equation}\label{eq:propagation}
 |X(a,t)-X(a,s)|\le V_*(t-s)\quad\text{for a.e. }a,\qquad
 \|X(t)\|_2\le\|X_0\|_2+\sqrt M V_*t.
\end{equation}
In particular $X\in C([0,T];L^2)$, and the second moments of $\rho_t$, $0\le t\le T$, are uniformly integrable. There is a unique narrowly continuous signed measure family $\qmeas_t$ agreeing with $u\rho_t$ for almost every $t$, with $\qmeas_0=u_0\rho_0$ and $|\qmeas_t|\le R\rho_t$. For $\phi\in C_c^1$,
\begin{equation}\label{eq:qtest}
 \frac{\dd}{\dd t}\int\phi\dd\qmeas_t=\int\phi'uf(u)\dd\rho_t
 \quad\text{a.e.},\qquad
 \left|\frac{\dd}{\dd t}\int\phi\dd\qmeas_t\right|
 \le MRV_*\|\phi'\|_\infty.
\end{equation}
\end{lemma}
\begin{proof}
Write $m_t(y)=\rho_t(({-}\infty,y])$. Let $\eta_\epsilon$ be smooth, nonincreasing, equal to one on $(-\infty,0]$ and zero on $[\epsilon,\infty)$. The test
$\eta_\epsilon(x-y-V_*(r-s))$ has material derivative $(b-V_*)\eta_\epsilon'\ge0$.
Multiply it by $\chi_N$, equal to one on $[-N,N]$, supported in $[-2N,2N]$, with $|\chi_N'|\le C/N$. The cutoff error is at most $CMV_*(t-s)/N$; time cutoffs and narrow continuity justify the endpoint evaluations. Letting $N\to\infty$, then $\epsilon\downarrow0$, and using the test with $+V_*(r-s)$ gives
\[
 m_t(y-V_*(t-s))\le m_s(y)\le m_t(y+V_*(t-s)).
\]
Generalized inversion proves the displacement bound. Minkowski's inequality gives the $L^2$ estimates, without trajectory uniqueness at zero. With $A_T=V_*T$ and $K>A_T$,
\begin{equation}\label{eq:mass-tail}
 \sup_{0\le t\le T}\int_{|x|>K}|x|^2\dd\rho_t(x)
 \le 2\int_{\{|X_0|>K-A_T\}}(|X_0|^2+A_T^2)\dd a
 \longrightarrow0.
\end{equation}
This gives uniformly integrable second moments and uniform tightness. If $\supp\rho_0\subset[A,B]$, finite displacement also gives
$\supp\rho_t\subset[A-V_*t,B+V_*t]$.

The momentum equation gives the distributional derivative in \eqref{eq:qtest}. Denote its Lipschitz representative by $L_t(\phi)$, with the prescribed initial value. Approximation by times at which $L_t(\phi)=\int\phi\,u\dd\rho_t$, together with narrow continuity of $\rho$, gives, for every $t$,
\[
 |L_t(\phi)|\le R\int|\phi|\dd\rho_t\le RM\|\phi\|_\infty.
\]
Linear identities extend from almost every time by continuity. Hence $L_t$ extends to $C_0(\R)$ and represents a unique signed Radon measure $\qmeas_t$. Uniform approximation by $C_c^1$ tests gives weak-star continuity on $C_0$ and $|\qmeas_t|\le R\rho_t$.
\begingroup
For $\psi\in C_b(\R)$, take $0\le\chi_N\le1$ as above. Domination gives, for $s,t\in[0,T]$,
\[
 \begin{aligned}
 \left|\int\psi\dd(\qmeas_t-\qmeas_s)\right|
 &\le\left|\int\chi_N\psi\dd(\qmeas_t-\qmeas_s)\right|\\
 &\quad+2R\|\psi\|_\infty
 \sup_{0\le r\le T}\rho_r(\{|x|>N\}).
 \end{aligned}
\]
For fixed $N$, the first term tends to zero as $t\to s$ by continuity on $C_0$. Letting $N\to\infty$ then removes the second term by \eqref{eq:mass-tail}. This proves continuity on $C_b$, including at $s=0$.
\endgroup
Agreement at almost every time gives uniqueness.
\end{proof}

For OSL transport and distinguished coefficient representatives, see \cite{BJtransport}. The common transport needed here is proved below from superposition.

\begin{proposition}[Deterministic common transport on positive-time strips]\label{prop:transport}
For every $0<s<T$ there is one ordered Borel path map $x\mapsto\Gamma_s(x)$ on a $\rho_s$-full set such that, with $\Phi_{s,t}(x)=\Gamma_s(x)(t)$, for every $s\le t\le T$,
\begin{equation}\label{eq:transport}
 \rho_t=\Phi_{s,t\#}\rho_s,\quad\qmeas_t=\Phi_{s,t\#}\qmeas_s,
 \quad0\le\Phi_{s,t}(y)-\Phi_{s,t}(x)\le(t/s)(y-x)
\end{equation}
for all $x<y$ in this $\rho_s$-full set. Paths have speed at most $V_*$ and cannot separate after meeting.
\end{proposition}
\begin{proof}
Superposition is applied to nonnegative combinations of the mass and momentum measures; uniqueness of ODE trajectories on positive-time strips, supplied by the Oleinik condition, then identifies a common deterministic transport map.

For $k>R$, the nonnegative narrowly continuous measures $\mu_t^\pm=k\rho_t\pm\qmeas_t$ satisfy
\[
 (k-R)\rho_t\le\mu_t^\pm\le(k+R)\rho_t,
 \qquad \partial_t\mu^\pm+\partial_x(b\mu^\pm)=0.
\]
Testing with $\chi_N$ gives constant positive mass: the flux error is bounded by $C(k+R)MV_*(T-s)/N$. After normalization, these are narrowly continuous probability solutions with Borel transport velocity and finite spacetime $2$-action, since $|b|\le V_*$ for $\mu_t^\pm\dd t$-almost every $(x,t)$. Their second moments are already controlled by \cref{lem:canonical}. Thus \cite[Theorem 8.2.1]{AGS} gives superposition measures on absolutely continuous ODE paths. Rescale them to the original masses.

For two such paths, the Oleinik condition \eqref{eq:OSL} and the positive-part chain rule give
\[
 \frac{\dd}{\dd t}(\gamma_2-\gamma_1)_+
 \le t^{-1}(\gamma_2-\gamma_1)_+,
 \qquad
 (\gamma_2(t)-\gamma_1(t))_+
 \le(t/s)(\gamma_2(s)-\gamma_1(s))_+.
\]
Exchanging the paths proves uniqueness from a common starting point. A first meeting therefore prevents crossing or subsequent separation. Only $s>0$ is used.

Let $\eta^\pm$ be the rescaled superposition measures. Disintegration over the initial evaluation on a measurable full set of ODE paths, and uniqueness for paths starting at $x$, give
\[
 \eta^\pm=\int\eta_x^\pm\,\dd\mu_s^\pm(x),\qquad
 \eta_x^\pm=\delta_{\Gamma_s^\pm(x)}
 \quad(\mu_s^\pm\text{-a.e. }x).
\]
This defines Borel path maps. Since $\mu_s^\pm\sim\rho_s$, uniqueness gives $\Gamma_s^+(x)=\Gamma_s^-(x)=:\Gamma_s(x)$ for $\rho_s$-a.e.\ $x$. Adding and subtracting their push-forwards proves \eqref{eq:transport}. ODE uniqueness also identifies restrictions at later source times and gives the composition law.
\end{proof}

\begin{theorem}[Lagrangian representation of entropy solutions]\label{thm:lift}
For every entropy solution of \textup{\cref{def:entropy}}, its quantile $X(t)=X_{\rho_t}$ and canonical Lagrangian physical velocity satisfy, on each finite time interval,
\begin{equation}\label{eq:entropy-representation}
 \begin{gathered}
 X\in C([0,T];L^2)\cap W^{1,\infty}(0,T;L^2),\qquad
 p_X(t)=P_{X(t)}U_0,\qquad X_t=f(p_X),\\
 \rho_t=X(t)_\#\dd a,\qquad\qmeas_t=X(t)_\#(U_0\dd a).
 \end{gathered}
\end{equation}
The identity $p_X(t)=P_{X(t)}U_0$ and the two push-forward identities hold at every time; the derivative identity holds in $L^2$ for almost every time. Moreover, $\|X(t)-X(s)\|_\infty\le V_*|t-s|$ for $0\le s,t\le T$, and the gap estimate \eqref{eq:limit-gap} holds at every positive time pair. The map has the no-splitting property from time zero. There is a common full-measure set of mass labels whose paths satisfy
\begin{equation}\label{eq:path}
 X(a,t)-X(a,s)=\int_s^tf(p_X(a,r))\dd r\qquad(0\le s<t)
\end{equation}
and whose label positions are continuity points of every $X(\cdot,t)$. Labels in this set are called \emph{regular labels}; all later refinements preserve full measure.
\end{theorem}
\begin{proof}
\emph{Conditional averaging at positive times.} For $0<s<t$, \cref{lem:quantile-details,prop:transport} give
\[
 X(t)=\Phi_{s,t}\circ X(s)\quad\text{a.e.}
\]
The gap and time-Lipschitz estimates follow from \cref{prop:transport}, as does the no-splitting property for $0<s<t$. Define the Lagrangian physical velocity $U(t)\in\mathcal H_{X(t)}$ through $X(t)_\#(U(t)\dd a)=\qmeas_t$. For bounded continuous $\zeta$, transport of the canonical momentum measure gives
\[
 \int\zeta(X(t))U(t)\dd a=\int\zeta(X(t))U(s)\dd a.
\]
Density of $C_c^\infty(\R)$ in $L^2(\rho_t)$, pulled back by $X(t)$, gives $U(t)=P_{X(t)}U(s)$.

\emph{Identification from the initial traces.} By \cref{lem:canonical}, $\|X(s)-X_0\|_2\le\sqrt M V_*s$. For any sequence $s_n\in G$ with $s_n\downarrow0$, pass to a subsequence such that $U(s_n)\rightharpoonup\overline U$ in $L^2$. For $\zeta\in C_c^\infty(\R)$, the functions $\zeta(X(s_n))$ converge strongly in $L^2$ to $\zeta(X_0)$, by their Lipschitz bound. The canonical momentum trace then yields
\[
 \int\zeta(X_0)\overline U\dd a=\int\zeta(X_0)U_0\dd a,
 \quad\text{so }P_{X_0}\overline U=U_0.
\]
The projection determines the average on each initial fibre. The kinetic energy trace, obtained by taking $\varphi\equiv1$ in \eqref{eq:E0}, excludes residual variation on these fibres, since
\[
 \|U_0\|_2=\|P_{X_0}\overline U\|_2
 \le\|\overline U\|_2\le\liminf_n\|U(s_n)\|_2=\|U_0\|_2.
\]
Orthogonal projection gives
$\|\overline U-U_0\|_2^2=\|\overline U\|_2^2-\|P_{X_0}\overline U\|_2^2=0$.
Since every weak subsequential limit is $U_0$, convergence of the norms gives strong $U(s)\to U_0$ as $s\downarrow0$ through $G$. For each fixed $t>0$, letting $s\downarrow0$ through $G$ in $U(t)=P_{X(t)}U(s)$ yields $U(t)=P_{X(t)}U_0$. Thus the canonical momentum measure has the push-forward representation in \eqref{eq:entropy-representation} at every time.

\emph{Common label paths and no-splitting at time zero.} For positive rational source times, pull the Borel path maps back by $X(s)$ and remove the null sets for all rational evaluations. Uniqueness on overlapping strips identifies these paths; their common speed bound gives continuous label paths on $(0,T]$. At every real time they represent $X(t)$ by rational approximation and $L^2$ continuity. Fubini and the pullback identity identify their derivative as $f(p_X)$ almost everywhere. Bounded speed extends the paths to zero, where an almost-everywhere subsequence of $X(s)\to X_0$ identifies the initial evaluation. Thus \eqref{eq:path} holds for all time pairs on one full-measure label set; intersecting over integer $T$ gives a global representative. Since $U_0\in\mathcal H_{X_0}$, \cref{lem:projection-approximation} applied to the identified formula $p_X(t)=P_{X(t)}U_0$ now gives $p_X(t)\to U_0$ strongly through all times.

For the no-splitting property at time zero, consider an initial positive fibre $J$ and write $X_0=x_0$ and $U_0=p$. The path identity gives
\[
 \left\|\frac{X(s)-x_0}{s}-f(p)\right\|_{L^2(J)}
 \le\frac{C_f}{s}\int_0^s\|p_X(r)-p\|_{L^2(J)}\dd r\longrightarrow0.
\]
The quotient is nondecreasing. At an interior label it lies between its averages on fixed left and right intervals, which both converge to $f(p)$ by $L^2$ convergence. Hence for interior regular labels $a<b$ in $J$,
$[X(b,s)-X(a,s)]/s\to0$. The positive-time gap estimate then gives
\[
 0\le X(b,t)-X(a,t)\le t\,[X(b,s)-X(a,s)]/s\longrightarrow0.
\]
Thus no initial fibre splits.

Finally intersect the path and Fubini sets with the simultaneous continuity-label set of \cref{lem:alltime-labels}.
\end{proof}

\begin{example}[The role of the initial kinetic energy trace]\label{ex:energy-needed}
Let $f(u)=u$, $\rho_0=\delta_0$, $u_0=0$. Besides the stationary mass, consider
\[
 \rho_t=\tfrac12\delta_{-t}+\tfrac12\delta_t,
 \qquad u(-t,t)=-1,\quad u(t,t)=1.
\]
Both paths solve the free ODE and give the prescribed weak initial mass and zero momentum trace. The Borel coefficient $b(x,t)=\max(-1,\min(1,x/t))$ satisfies the Oleinik condition \eqref{eq:OSL}. The total kinetic energy, however, equals one for every $t>0$, whereas its initial value is zero. Thus \eqref{eq:E0}, tested with $\varphi\equiv1$, excludes this splitting evolution from the entropy class.
\end{example}

\subsection[\texorpdfstring{Local BV regularity of the Lagrangian physical velocity}{Local BV regularity of the Lagrangian physical velocity}]{Local BV regularity of the Lagrangian physical velocity}\label{sec:state-representatives}

Call a time $t>0$ \emph{structural} if \eqref{eq:OSL} holds for all $x<y$, $\qmeas_t=u(\cdot,t)\rho_t$, and the pullback and path-velocity identities hold for almost every label. These properties hold on a Borel full-measure time set $\mathcal G_X$; for two flows use $\mathcal G_X\cap\mathcal G_Y$. At a structural restart $\tau$, refine the regular labels by the Fubini set for $P_{X(t)}p_X(\tau)$ and the Lebesgue points of $p_X(\tau)$.

\begin{lemma}[Local BV regularization and finite position traces]\label{lem:BV}\label{lem:finite-trace-localization}
At almost every $t>0$, on each label interval $J$ with $\osc_JX(t)<\infty$,
\begin{equation}\label{eq:BV}
 \TV_J(f(p_X(t)))\le2V_*+2\osc_JX(t)/t,\qquad
 \TV_J(p_X(t))\le c_f^{-1}\TV_J(f(p_X(t))).
\end{equation}
In particular the physical velocities belong to $BV_{\rm loc}(0,M)$. Fix a structural time $\tau>0$. If $\osc_JX_0<\infty$, then for every $t\ge\tau$,
\begin{equation}\label{eq:local-BV-restart}
 \TV_J(p_X(t))\le\TV_J(p_X(\tau))+4R.
\end{equation}
At a finite one-sided position trace $X(c\pm,t)$, $t\ge\tau$, there is a neighbourhood on that side where $X_0$ is bounded and $p_X(\tau)$ is $BV$. Its one-sided essential oscillation about the physical velocity trace tends to zero. This includes finite inward traces at $0$ or $M$; no physical velocity trace is asserted at an infinite position endpoint.
\end{lemma}
\begin{proof}
At a time when the Oleinik condition and the pullback identities hold, $B=f(p_X(t))$ satisfies
$B(b)-B(a)\le[X(b,t)-X(a,t)]/t$ for almost every ordered pair; equal positions have equal physical velocities. The static estimate in \cref{lem:OSL-BV} gives \eqref{eq:BV} and a locally regulated representative. The no-splitting and tower properties give
\begin{equation}\label{eq:restart}
 p_X(t)=P_{X(t)}U_\tau,\qquad U_\tau=p_X(\tau),\quad t\ge\tau.
\end{equation}
If $\osc_JX_0<\infty$, finite displacement gives
$\osc_JX(\tau)\le\osc_JX_0+2V_*\tau$. Thus $U_\tau|_J\in BV(J)$, and \eqref{eq:local-BV-projection} proves \eqref{eq:local-BV-restart}.

For $z=X(c+,t)\in\R$, choose an internal regular label $d>c$. Monotonicity and finite displacement give
\[
 z-V_*t\le X_0(a)\le X_0(d)\quad(c<a<d)
\]
almost everywhere. Hence $U_\tau$ is $BV$ up to $c$ on the right, also for $c=0$. For $z=X(c-,t)\in\R$, choose $d<c$ and use
\[
 X_0(d)\le X_0(a)\le z+V_*t\quad(d<a<c)
\]
to obtain the left trace, also for $c=M$. A bounded $BV$ representative has a one-sided limit with vanishing essential oscillation about it. Internal position traces are finite because $X(t)\in\K\subset L^2$. The local projection estimate gives the corresponding later physical velocity representatives. Finite displacement preserves infinite endpoint positions, where this argument asserts no physical velocity trace.
\end{proof}

\begin{lemma}\label{lem:trace-Fubini}
Fix a structural restart $\tau>0$ and the Lagrangian map $X$ of an entropy solution. There is a fixed full-measure refinement $A\subset(0,M)$ of the regular labels such that: \textup{(a)} \eqref{eq:path} holds for every $a\in A$ and every time pair; \textup{(b)} $X(a-,t)=X(a+,t)$ for $(a,t)\in A\times[0,\infty)$; \textup{(c)} for each $a\in A$, outside a null set $N_a\subset(\tau,\infty)$, $p_X(a,r)=p_X^+(a,r)$, where $p_X^+$ is a jointly Borel, locally regulated representative that is right-continuous in the label variable, and this value is the appropriate fibre average or singleton value from \eqref{eq:restart}. For any interior cut $c$, or endpoint with finite position trace, every countable sequence $a_n\in A$ approaching from one side satisfies
\[
 p_X(a_n,r)\longrightarrow p_X(c\pm,r)\quad(a_n\to c\pm),
 \qquad r\in(\tau,\infty)\setminus\bigcup_nN_{a_n}.
\]
\end{lemma}
\begin{proof}
Assertions \textup{(a)}--\textup{(b)} hold on the label set of \cref{thm:lift}. Choose the jointly Borel Lagrangian physical velocity of \cref{lem:joint}. By \cref{lem:BV}, each slice for $r\ge\tau$ is $BV$ on a countable exhaustion by compact interior intervals, augmented by $(0,M/2)$ or $(M/2,M)$ when the corresponding endpoint position is finite. Its right representative is
\[
 p_X^+(a,r)=\lim_{n\to\infty}
 \frac1{\epsilon_n(a)}\int_a^{a+\epsilon_n(a)}p_X(b,r)\dd b,
 \qquad \epsilon_n(a)=\min\{1/n,(M-a)/2\}.
\]
Set an arbitrary bounded value where the limit fails. This is jointly Borel and agrees with $p_X$ almost everywhere in each slice. Finite endpoint values are supplied by the interior limits; no value at an infinite endpoint is used as a trace. Fubini, applied also to \eqref{eq:restart} with the fibre/singleton convention, gives a fixed full-measure refinement $A$ and, for each $a\in A$, a null exceptional time set $N_a$ satisfying \textup{(c)}.

For a specified countable sequence $(a_n)\subset A$, the set $N_{(a_n)}=\bigcup_{n\ge1}N_{a_n}$ is null. At each $r>\tau$ outside this union, the sequence eventually lies in one of the above intervals near the selected cut, where its physical velocities converge to the regulated one-sided trace. Boundedness then permits dominated convergence in the path integrals of \eqref{eq:trace-integral}. The refinement $A$ is fixed before the sequence; only $N_{(a_n)}$ may depend on it.
\end{proof}

\subsection[Compression inequalities for \texorpdfstring{fibre averages}{fibre averages}]{Compression inequalities for fibre averages}\label{sec:calibration}

Averages of physical velocity in a terminal fibre are compared using cumulative mass and momentum evaluated along affine lines. For fixed $\theta\in\R$, consider the line $\gamma(r)=y+(r-s)f_R(\theta)$; it need not be a characteristic of the solution. Set
\[
 m(x,t)=\rho_t(({-}\infty,x]),\quad
 q(x,t)=\qmeas_t(({-}\infty,x]),\quad
 J_\theta=q-\theta m.
\]
These cumulative quantities are locally $BV$ in space. We denote the strict half-line traces by $m^-,q^-$ and the non-strict traces by $m^+,q^+$.

The following estimate gives time continuity on the whole line. Only differences of cumulative functions are asserted to belong to $L^1(\R)$.

\begin{lemma}[Time regularity of cumulative mass and momentum]\label{lem:cumulative-current}
For all $s,t\ge0$,
\begin{equation}\label{eq:cumulative-time}
 \|m(\cdot,t)-m(\cdot,s)\|_{L^1(\R)}\le MV_*|t-s|,
 \qquad
 \|q(\cdot,t)-q(\cdot,s)\|_{L^1(\R)}\le RMV_*|t-s|.
\end{equation}
In spacetime distributions,
\begin{equation}\label{eq:cumulative-current}
 \partial_tm=-f(u)\rho,\qquad \partial_tq=-uf(u)\rho,
 \qquad\partial_xm=\rho,\qquad\partial_xq=u\rho.
\end{equation}
\end{lemma}
\begin{proof}
At every time,
\[
 m(x,t)=\int_0^M\one_{\{X(a,t)\le x\}}\dd a,
 \qquad q(x,t)=\int_0^MU_0(a)\one_{\{X(a,t)\le x\}}\dd a.
\]
The $L^1(\R)$ distance between $\one_{[z,\infty)}$ and $\one_{[w,\infty)}$ is $|z-w|$; Fubini, $|U_0|\le R$ and finite displacement prove \eqref{eq:cumulative-time}.

For $\phi\in C_c^\infty(\R\times(0,T))$, apply the chain rule to
$\Psi(X(a,t),t)$, where $\Psi(z,t)=\int_z^\infty\phi(x,t)\dd x$. The function and its derivatives are bounded, and compact time support gives
\[
 \int_0^T\!\int_{X(a,t)}^\infty\phi_t(x,t)\dd x\dd t
 =\int_0^T\phi(X(a,t),t)X_t(a,t)\dd t.
\]
Integrating over labels proves the first time identity. Multiplication by $U_0$ proves the second, since $X_t=f(p_X)$ and conditional expectation replaces $U_0$ by $p_X$ against $\phi(X)f(p_X)\in\mathcal H_X$. The spatial identities follow from the indicators. No constant at spatial infinity is introduced.
\end{proof}

\begin{lemma}[Monotonicity along affine lines]\label{lem:pencil}
For each fixed real $\theta$,
\begin{equation}\label{eq:current}
 (\partial_t+f_R(\theta)\partial_x)J_\theta
 =-\W(u,\theta)\rho,
 \qquad\W(r,\theta)=(r-\theta)(f_R(r)-f_R(\theta))\ge0.
\end{equation}
Moreover
\begin{equation}\label{eq:W}
 \W(r,\theta)=D_H(r\mid\theta)+D_H(\theta\mid r),\qquad
 c_f|r-\theta|^2\le\W(r,\theta)\le C_f|r-\theta|^2
\end{equation}
for $r\in[-R,R]$ and $\theta\in\R$. For all $0\le s<t$ and $y\in\R$,
\begin{equation}\label{eq:pencil}
 J_\theta^\pm(y+(t-s)f_R(\theta),t)\le J_\theta^\pm(y,s).
\end{equation}
For each $\theta\in\R$, the inequality holds for every time pair with the indicated one-sided traces.
\end{lemma}
\begin{proof}
By \cref{lem:cumulative-current}, the coefficient of $\rho$ on the left of \eqref{eq:current} is
$-uf(u)+\theta f(u)+f_R(\theta)(u-\theta)=-\W(u,\theta)$.
Here $f_R(u)=f(u)$ on the physical velocity range. Adding the Bregman divergences and using the derivative bounds on $f_R$ proves \eqref{eq:W}.

The shear $\widetilde J(z,t)=J_\theta(z+tf_R(\theta),t)$ satisfies $\partial_t\widetilde J\le0$ in distributions.
\begingroup
Choose $0\le\kappa\in C_c^\infty(0,1)$ with $\int\kappa=1$, set $\kappa_\epsilon(w)=\epsilon^{-1}\kappa(w/\epsilon)$, and define
\[
 F_\epsilon^\pm(z,t)=\int_0^\epsilon
 \kappa_\epsilon(w)\widetilde J(z\pm w,t)\dd w.
\]
The indicator calculation above gives
\endgroup
\[
 \|m(\cdot+h,t)-m(\cdot,t)\|_1\le M|h|,\qquad
 \|q(\cdot+h,t)-q(\cdot,t)\|_1\le RM|h|
\]
\begingroup
and, together with \eqref{eq:cumulative-time}, yields
\[
 \begin{aligned}
 \|\widetilde J(\cdot,t)-\widetilde J(\cdot,s)\|_1
 &\le M(R+|\theta|)(V_*+|f_R(\theta)|)|t-s|,\\
 |F_\epsilon^\pm(z,t)-F_\epsilon^\pm(z,s)|
 &\le\|\kappa_\epsilon\|_\infty
 \|\widetilde J(\cdot,t)-\widetilde J(\cdot,s)\|_1.
 \end{aligned}
\]
Since $|\widetilde J|\le M(R+|\theta|)$, spatial convolution also gives continuity in $z$, uniformly in $t$ for fixed $\epsilon$. Thus $F_\epsilon^\pm$ is jointly continuous, including at time zero. Convolution preserves $\partial_t\widetilde J\le0$, so for $0\le\psi\in C_c^\infty(\R)$ the continuous function $t\mapsto\int\psi F_\epsilon^\pm\dd z$ is nonincreasing. Consequently, for every $0\le s<t$,
\[
 \int\psi(z)[F_\epsilon^\pm(z,t)-F_\epsilon^\pm(z,s)]\dd z\le0.
\]
Continuity in $z$ gives $F_\epsilon^\pm(z,t)\le F_\epsilon^\pm(z,s)$ for every $z$. At each fixed time, the one-sided $BV$ limits give
\[
 F_\epsilon^\pm(z,t)\longrightarrow
 J_\theta^\pm(z+tf_R(\theta),t)
 \qquad(\epsilon\downarrow0).
\]
Taking this limit at $s$ and $t$, and setting $z=y-sf_R(\theta)$, proves \eqref{eq:pencil} for every time pair and every fixed real $\theta$.
\endgroup
\end{proof}

In mass coordinates, the source traces are expressed as follows. For $0\le s<t$, set
\begin{equation}\label{eq:source}
 Z_s=X(s),\qquad U_s=P_{X(s)}U_0,\qquad
 Q_s(a)=\int_0^aU_s(b)\dd b.
\end{equation}
At $s=0$, we have $U_s=U_0$. The no-splitting property and the tower property give
$p_X(t)=P_{X(t)}U_s$ and $\qmeas_t=X(t)_\#(U_s\dd a)$, the latter by testing against bounded functions of $X(t)$. Define
\[
 \lambda_s(y)=|\{a:Z_s(a)<y\}|,\qquad
 \varrho_s(y)=|\{a:Z_s(a)\le y\}|.
\]
The strict and non-strict source traces are then
\begin{equation}\label{eq:dictionary}
 (m_s^-,q_s^-)(y)=(\lambda_s(y),Q_s(\lambda_s(y))),\quad
 (m_s^+,q_s^+)(y)=(\varrho_s(y),Q_s(\varrho_s(y))).
\end{equation}
The two inverse images are prefixes of the mass-label interval $(0,M)$. At a vacuum jump cut $c$ both cuts equal $c$; at an atom fibre $(\ell,r)$ the strict and non-strict cuts are $\ell$ and $r$.

\begin{theorem}[Source-time compression at every cut]\label{thm:allcut}
Let $I=(\alpha,\beta)$ be a maximal positive fibre of $X(t)$, at position $x$, and let
\[
 p_I=\mean{I}{U_0}=\mean{I}{U_s}.
\]
Then for every $c\in[\alpha,\beta]$,
\begin{equation}\label{eq:allcut}
 \int_\alpha^c(U_s-p_I)\dd a\ge0,\qquad
 \int_c^\beta(U_s-p_I)\dd a\le0.
\end{equation}
The source time may be $s=0$ or any positive time preceding $t$.
\end{theorem}
\begin{proof}
The proof reduces compression to nonnegativity of a scalar defect. The inequalities in \eqref{eq:pencil} control source cuts; continuity and affinity then cover the endpoints and interiors of source atoms.

The no-splitting property, including at zero, makes $I$ a union of whole source fibres; conditional expectation identifies its two means. Set
\[
 \Delta=t-s,\quad y_I=x-\Delta f(p_I),\quad
 F(c)=Q_s(c)-Q_s(\alpha)-p_I(c-\alpha).
\]
Then
\begin{equation}\label{eq:cut-defect}
 F(\alpha)=F(\beta)=0,\qquad
 F(c)=p_I(\beta-c)-[Q_s(\beta)-Q_s(c)],
\end{equation}
so both assertions reduce to $F(c)\ge0$. On $I$, finite displacement gives $|Z_s(a)-x|\le V_*\Delta$; all source positions and inward traces used below are finite, also when $\alpha=0$ or $\beta=M$.

At the terminal atom,
\[
 (m_t^-,q_t^-)(x)=(\alpha,Q_s(\alpha)),\qquad
 (m_t^+,q_t^+)(x)=(\beta,Q_s(\beta)).
\]
For $y\in\R$, put $\theta(y)=g((x-y)/\Delta)$, using the auxiliary inverse. Applying \eqref{eq:pencil} with each source trace gives
\begin{align}
 F(\lambda_s(y))&\ge(\theta(y)-p_I)(\lambda_s(y)-\alpha),\label{eq:cutleft}\\
 F(\varrho_s(y))&\ge(p_I-\theta(y))(\beta-\varrho_s(y)).\label{eq:cutright}
\end{align}
Use these only when the indicated source cut belongs to $[\alpha,\beta]$.

\emph{Non-atomic cuts and vacuum gaps.} If $Z_s(c-)=Z_s(c+)=z$ and $\rho_s(\{z\})=0$, take $y=z$; if $Z_s(c-)<Z_s(c+)$, take any $y$ in the open gap. In either case $\lambda_s(y)=\varrho_s(y)=c$. The right-hand side of \eqref{eq:cutleft} is nonnegative when $\theta(y)\ge p_I$, and that of \eqref{eq:cutright} is nonnegative when $\theta(y)\le p_I$. Thus $F(c)\ge0$ in both cases.

\emph{Endpoints of source atoms.} Let $J=(\ell,r)\subset I$ be a positive source fibre at position $z$. If $\ell=\alpha$, then $F(\ell)=0$. If $\ell>\alpha$ and $z\le y_I$, use the strict source trace and \eqref{eq:cutleft}. If $\ell>\alpha$ and $z>y_I$, choose non-atomic points $y_n\uparrow z$ with $y_n>y_I$, possible since the atoms are countable. Then $c_n=m_s(y_n)\uparrow\ell$ lies in $I$ eventually, both source cuts equal $c_n$, and $\theta(y_n)<p_I$. Equation \eqref{eq:cutright} gives $F(c_n)\ge0$; continuity gives $F(\ell)\ge0$.

If $r=\beta$, then $F(r)=0$. For $r<\beta$ and $z\ge y_I$, use the non-strict trace and \eqref{eq:cutright}. For $r<\beta$ and $z<y_I$, choose non-atomic $y_n\downarrow z$, $y_n<y_I$. The cuts $c_n=m_s(y_n)\downarrow r$ lie in $I$ eventually; \eqref{eq:cutleft} and continuity give $F(r)\ge0$.

\emph{Interiors of source atoms.} On $J$, $U_s$ is constant and $F$ is affine, with nonnegative endpoint values. Hence $F\ge0$ on $J$. These cases exhaust $[\alpha,\beta]$.
\end{proof}

\begin{corollary}[Obstacle and normal-cone relation]\label{cor:obstacle}
For $t\ge s$ put $\Xi_s(a,t)=Q_s(a)-\int_0^ap_X(b,t)\dd b$. Then
\begin{equation}\label{eq:obstacle}
 \Xi_s\ge0,\quad\Xi_s(0,t)=\Xi_s(M,t)=0,\quad
 \Xi_s(a,t)=0\quad(a\in(0,M)\setminus\Omega_{X(t)}),
\end{equation}
and $U_s-p_X(t)\in N_\K(X(t))$.
\end{corollary}
\begin{proof}
Decompose $(0,a)$ into complete positive fibres of $X(t)$, the non-flat set and, when present, a final positive-fibre prefix. The defect vanishes on complete fibres by averaging, and on the non-flat set because $P_{X(t)}U_s=U_s$ by \cref{lem:PX}. The remaining prefix contributes a nonnegative defect by \eqref{eq:allcut}. Thus $\Xi_s\ge0$ and vanishes outside positive-fibre interiors. Its continuity gives zero integral against $D_aX(t)$, and \cref{lem:normal} yields the normal-cone relation. Taking the source at zero gives $U_0-p_X\in N_\K(X)$ for the entropy solution under consideration.
\end{proof}
\subsection[\texorpdfstring{One-sided physical velocities and position traces}{One-sided physical velocities and position traces}]{One-sided physical velocities and position traces}\label{sec:cut-curvature}

For a cut $c$ outside positive-fibre interiors, write $L_c(t)=X(c-,t)$ and $R_c(t)=X(c+,t)$. Internal traces are finite; endpoint traces $R_0,L_M$ are used only when finite. Finite displacement makes their finiteness independent of time and gives the $V_*$-Lipschitz bound. The local BV representatives and finite one-sided physical velocity traces are supplied by \cref{lem:BV}; no physical velocity trace is used at an infinite position endpoint. Fix a structural restart $\tau>0$ and write $U_\tau=P_{X(\tau)}U_0$. Recall that $Q_U(a)=\int_0^aU_0(b)\dd b$ is the primitive of the initial physical velocity.

At a given time, a side of a cut is called \emph{exposed} if no positive fibre contains a one-sided neighbourhood of the cut on that side. This does not imply vacuum and does not exclude small fibres accumulating at the cut.

\begin{lemma}[Identification of exposed one-sided physical velocities]\label{lem:Dini}
If at $t\ge\tau$ no positive fibre contains a right neighbourhood of a cut $c$, then
\begin{equation}\label{eq:DiniR}
 p_X(c+,t)=U_\tau(c+)=\underline D_+Q_U(c)
 :=\liminf_{h\downarrow0}\frac{Q_U(c+h)-Q_U(c)}h.
\end{equation}
If no positive fibre contains a left neighbourhood, then
\begin{equation}\label{eq:DiniL}
 p_X(c-,t)=U_\tau(c-)=\overline D_-Q_U(c)
 :=\limsup_{h\downarrow0}\frac{Q_U(c)-Q_U(c-h)}h.
\end{equation}
The interior-only formulas hold at $0$ and $M$ whenever the corresponding position trace is finite. A positive fibre starting at $c$ satisfies $p_I\le\underline D_+Q_U(c)$; one ending at $c$ satisfies $\overline D_-Q_U(c)\le p_I$.
\end{lemma}
\begin{proof}
\emph{Localization of the Lagrangian physical velocity.} Consider an exposed right side. For every $\eta>0$, all fibres meeting $(c,c+\epsilon_\eta)$ lie in $(c,c+\eta)$ for some $\epsilon_\eta>0$: otherwise fibres meeting arbitrarily close to $c$ and extending beyond $c+\eta$ would overlap and hence coincide, producing a positive right fibre. By \cref{lem:PX} and the tower property,
\begin{equation}\label{eq:exposed-essential-modulus}
 \esssup_{c<a<c+\epsilon_\eta}|p_X(a,t)-U_\tau(c+)|
 \le\esssup_{c<b<c+\eta}|U_\tau(b)-U_\tau(c+)|\longrightarrow0.
\end{equation}
This identifies the essential trace; \cref{lem:trace-Fubini} identifies it along any selected countable set of regular paths almost everywhere in time.

\emph{Identification of the Dini slopes.} Since $\Xi_0(c,t)=0$, the obstacle relation \eqref{eq:obstacle} gives
\[
 Q_U(c+h)-Q_U(c)\ge\int_c^{c+h}p_X(a,t)\dd a.
\]
There are cuts $c_n\downarrow c$ outside positive-fibre interiors, since otherwise a right neighbourhood would lie in one positive fibre. At these cuts equality holds, so the corresponding difference quotients converge to $p_X(c+,t)$. This proves \eqref{eq:DiniR}. On the left the same obstacle relation reads
$Q_U(c)-Q_U(c-h)\le\int_{c-h}^cp_X(a,t)\dd a$; equality along cuts approaching from the left gives \eqref{eq:DiniL}. Prefix and suffix compression on a positive fibre give the two asserted Dini inequalities. At $0,M$ only the available finite-position side is used.
\end{proof}

\begin{proposition}[Order of one-sided physical velocities and convexity of position traces]\label{thm:cut-curvature}
Suppose $c$ lies outside positive-fibre interiors on $[s,T]$ (or on $[s,T)$ with endpoint limits). Set $p_\pm(t)=p_X(c\pm,t)$. For structural times $r<t$ in this interval,
\[
 p_-(r)\le p_-(t),\qquad p_+(r)\ge p_+(t).
\]
The traces satisfy
\begin{equation}\label{eq:trace-integral}
 L_c(t)-L_c(r)=\int_r^t f(p_-(q))\dd q,\qquad
 R_c(t)-R_c(r)=\int_r^t f(p_+(q))\dd q.
\end{equation}
Consequently $L_c$ is convex and $R_c$ is concave. The one-sided endpoint assertions hold for $R_0$ or $L_M$ when that trace is finite.
\end{proposition}
\begin{proof}
Compression and identification of exposed physical velocity traces give the order of physical velocities at structural times. Passing to one-sided limits in the path identity yields the trace integrals and curvature.

Fix structural times $r<t$. A positive left fibre $I_r=(\alpha_r,c)$ is a suffix of its descendant $I_t=(\alpha_t,c)$, with $\alpha_t\le\alpha_r$; a positive right fibre $J_r=(c,\beta_r)$ is a prefix of its descendant $J_t=(c,\beta_t)$, with $\beta_r\le\beta_t$. Source-time compression gives, on the respective positive sides,
\[
 p_-(r)=\frac1{|I_r|}\int_{I_r}U_r\dd a\le p_-(t),\qquad
 p_+(r)=\frac1{|J_r|}\int_{J_r}U_r\dd a\ge p_+(t).
\]
If a side is exposed at $r$ and positive at $t$, apply suffix compression to $(c-h,c)$ on the left, or prefix compression to $(c,c+h)$ on the right, and let $h\downarrow0$. If it remains exposed, \cref{lem:Dini} with restart $r$ gives $p_\pm(t)=U_r(c\pm)=p_\pm(r)$ on that side.

These orders define monotone time representatives on the structural set. Choose regular labels $a_n\uparrow c$ in the path identity. By \cref{lem:trace-Fubini}, $p_X(a_n,q)\to p_-(q)$ for almost every intermediate $q$. Dominated convergence gives the first identity in \eqref{eq:trace-integral}; labels $a_n\downarrow c$ give the second. The monotonicity of $f$ now implies convexity of $L_c$ and concavity of $R_c$. Their Lipschitz continuity extends both conclusions to interval endpoints, including time zero when the cut stays outside fibre interiors, and to each finite domain-endpoint trace.
\end{proof}

\subsection{Backward fans and cumulative identities}\label{sec:affine-fans}\label{sec:fan}

Characteristic triangles for the nonlinear system already appear in \cite{Huang}. Here source-time compression gives the curvature of one-sided position traces and the affine separators of terminal fibres for every entropy solution. The full fan and cumulative identities below are structural consequences; the later comparison uses the one-sided physical velocity information in \cref{lem:Dini}.

\begin{lemma}[Affine bounds and rigidity of interior equality]\label{lem:chord}
Let $L$ be convex, $R$ concave and $L\le R$ on $[s,T]$. If $y\in[L(s),R(s)]$, $x\in[L(T),R(T)]$, then
\begin{equation}\label{eq:chord}
 \gamma(r)=\frac{T-r}{T-s}y+\frac{r-s}{T-s}x
 \quad\text{satisfies}\quad L(r)\le\gamma(r)\le R(r).
\end{equation}
If $L(r_0)=\gamma(r_0)$ or $R(r_0)=\gamma(r_0)$ at one interior time, that bounding trace equals $\gamma$ throughout $[s,T]$.
\end{lemma}
\begin{proof}
The convex and concave chord inequalities give \eqref{eq:chord}. The functions $\gamma-L$ and $R-\gamma$ are nonnegative and concave; an interior zero forces either function to vanish identically, proving rigidity.
\end{proof}

\begin{theorem}[Affine backward fan]\label{thm:fan}
Let $I=(\alpha,\beta)$ be a maximal positive fibre at time $T>0$ and position $x$. Define
\begin{align}
 y_-&=X_0(\alpha+),&\gamma_-(t)&=y_-+\frac tT(x-y_-),\label{eq:fan-minus}\\
 y_+&=X_0(\beta-),&\gamma_+(t)&=y_++\frac tT(x-y_+).\label{eq:fan-plus}
\end{align}
These inner initial traces are finite, including at a domain endpoint. Their slopes belong to $[v_-,v_+]$, $\gamma_-\le\gamma_+$, and for every $0\le t\le T$,
\begin{equation}\label{eq:fan-separation}
 \begin{cases}
 X(a,t)\le\gamma_-(t),&a<\alpha,\\
 \gamma_-(t)\le X(a,t)\le\gamma_+(t),&\alpha<a<\beta,\\
 X(a,t)\ge\gamma_+(t),&a>\beta,
 \end{cases}
\end{equation}
in the appropriate monotone-trace sense. If either side meets its corresponding bounding position trace at an interior time, then that trace coincides with the side throughout $[0,T]$.
\end{theorem}
\begin{proof}
For almost every $a\in I$, $|X_0(a)-x|\le V_*T$. Thus monotonicity makes $y_-,y_+$ finite, including at domain endpoints. Earlier fibres cannot cross $\alpha$ or $\beta$, since the no-splitting property would contradict maximality of $I$ at $T$. At internal boundary cuts, $L_\alpha(T)\le R_\alpha(T)=x$ and $L_\beta(T)=x\le R_\beta(T)$.
Apply \cref{thm:cut-curvature,lem:chord} with initial points $R_\alpha(0)$ and $L_\beta(0)$. At $\alpha=0$, concavity of $R_0$ supplies the lower inner bound; at $\beta=M$, convexity of $L_M$ supplies the upper inner bound. The missing exterior label sets are empty.

For inner labels, $R_\alpha(t)\le X(a,t)\le L_\beta(t)$; exterior labels satisfy $X(a,t)\le L_\alpha(t)$ or $X(a,t)\ge R_\beta(t)$. The chord bounds therefore give \eqref{eq:fan-separation}. Monotonicity of $X_0$ orders $y_-,y_+$, hence the sides sharing terminal point $x$. Approaching the two boundaries from inside $I$, the quotients $[x-X_0(a)]/T$ are time averages of transport velocities in $[v_-,v_+]$; their limits give the slope bounds. Interior equality is the rigidity assertion of \cref{lem:chord}.
\end{proof}

\begin{proposition}[Restriction to whole fibres]\label{prop:restriction}
Suppose a fixed label interval $A$ is a union of whole fibres of $X(t)$ for every $t$ in an open interval $(s,T)$. Define
\[
 \rho_t^A=X(t)_\#(\one_A\dd a),\qquad
 \qmeas_t^A=X(t)_\#(U_0\one_A\dd a).
\]
With the physical velocity $u$ of the unrestricted solution on $\rho_t^A\dd t$, they satisfy in $\mathcal D'(\R\times(s,T))$
\[
 \partial_t\rho^A+\partial_x(f(u)\rho^A)=0,\qquad
 \partial_t\qmeas^A+\partial_x(f(u)\qmeas^A)=0.
\]
Their totals are $\rho_t^A(\R)=|A|$ and $\qmeas_t^A(\R)=\int_AU_0\dd a$. No lateral source is produced by the restriction.
\end{proposition}
\begin{proof}
Since $\one_A$ is constant on every positive fibre of $X(t)$, \cref{lem:PX} gives $\one_A\in\mathcal H_{X(t)}$. Hence
\[
 P_{X(t)}(U_0\one_A)=\one_Ap_X(t).
\]
Integrating the path chain rule with the fixed weights $\one_A$ and $U_0\one_A$, and using this identity in the latter case, gives the two weak equations. Dominated convergence along continuous paths gives narrow temporal traces. The totals are the integrals of these fixed weights.
\end{proof}

\begin{proposition}[One-sided cumulative identities]\label{thm:raw-sides}
For the fan in \textup{\cref{thm:fan}}, for every $0<t<T$,
\begin{align}
 m^-(\gamma_-(t),t)&=\alpha,&q^-(\gamma_-(t),t)&=Q_U(\alpha),\label{eq:raw-minus}\\
 m^+(\gamma_+(t),t)&=\beta,&q^+(\gamma_+(t),t)&=Q_U(\beta).\label{eq:raw-plus}
\end{align}
Thus the lower strict trace excludes the whole fan mass, while the upper non-strict trace includes it. The formulas also hold at interior merger times.
\end{proposition}
\begin{proof}
By \eqref{eq:fan-separation}, lower exterior labels lie at or below $\gamma_-$, and inner labels at or above it. If, for some $t_0\in(0,T)$,
\[
 E=\{a<\alpha:X(a,t_0)=\gamma_-(t_0)\},\qquad |E|>0,
\]
then $L_\alpha(t_0)=\gamma_-(t_0)$ and \cref{lem:chord} gives $L_\alpha\equiv\gamma_-$. Monotonicity and the no-splitting property keep this positive fibre, modulo endpoints, on the left trace, so $X(a,T)=L_\alpha(T)=x$ for a.e. $a\in E$, contradicting maximality of $I$. Thus $\{X<\gamma_-\}=(0,\alpha)$ modulo null sets, and push-forward of $1,U_0$ gives \eqref{eq:raw-minus}; for $\alpha=0$ the strict lower set is empty.

Positive mass above $\beta$ on $\gamma_+$ gives $R_\beta(t_0)=\gamma_+(t_0)$; concave rigidity and no-splitting retain this mass at $x$ at time $T$, again contradicting maximality. Hence $\{X\le\gamma_+\}=(0,\beta)$ modulo null sets, proving \eqref{eq:raw-plus}; for $\beta=M$ there is no exterior mass. All-time path continuity includes interior merger times.
\end{proof}

\section[Quantitative stability and uniqueness]{Quantitative stability and uniqueness}\label{sec:compare}

We compare the Lagrangian maps of two arbitrary entropy solutions. Local averages of physical velocity and short-time expansions provide the cluster comparison; an estimate for the essential supremum of the position gap then yields $W_\infty$ stability of the normalized mass measures and equal-data uniqueness.

\subsection[\texorpdfstring{Local physical velocity averages and cluster contacts}{Local physical velocity averages and cluster contacts}]{Local physical velocity averages and cluster contacts}\label{sec:localization}\label{sec:contact-times}

The estimates in this subsection prepare the short-time expansions of a single flow; they will be used to compare two flows with different cluster partitions.

Let $Z$ be the Lagrangian map of an entropy solution. Fix a structural restart $\tau>0$, put $U_\tau=P_{Z(\tau)}U_0\in BV_{\mathrm{loc}}(0,M)$, and use the finite-position physical velocity traces of \cref{lem:BV}. The tower property gives
\begin{equation}\label{eq:oneflow-tower}
 p_Z(t)=P_{Z(t)}U_\tau\quad(t\ge\tau).
\end{equation}

The restart profile $U_\tau$ supplies the local one-sided traces used in the short-time expansions. For two flows, their restart profiles need not remain as close as the initial physical velocities. The quantitative comparison will therefore use the initial profiles $U_0,W_0$ through source-time compression at zero.

Write $I_Z(a,t)$ for the closure in $[0,M]$ of the level set of $Z(\cdot,t)$ containing a regular label $a$. It is either a singleton or the closure of a positive fibre; in the latter case, endpoints do not affect averages. The physical velocity $p_Z(a,t)$ is the average of $U_\tau$ over $I_Z(a,t)$ when this fibre has positive mass, and equals $U_\tau(a)$ when it is a singleton. By \cref{lem:trace-Fubini}, this holds along paths on a fixed full-measure label set for almost every later time; a selected countable label sequence admits a common such time set.

\begin{lemma}[Localization on shrinking time intervals]\label{lem:localization}
Fix $t\ge\tau$ and $h_n\downarrow0$.
\begin{enumerate}[label=\textup{(\roman*)}]
\item If $I$ is a maximal positive fibre at $t$, its descendants $I(q)$ satisfy
\begin{equation}\label{eq:etaI}
 \eta_I(h):=\sup_{0\le q\le h}|I(q)\setminus I|\longrightarrow0.
\end{equation}
\item Suppose there is no positive right fibre at a cut $c$ with finite $Z(c+,t)$, let regular labels $a_n\downarrow c$, and let $S_n\subset[0,h_n]$ be arbitrary sets. If the fibres $J_n(q)=I_Z(a_n,t+q)$ contain no labels strictly less than $c$ for $q\in S_n$, then
\begin{equation}\label{eq:stopped-loc}
 \sup_{q\in S_n}(\sup J_n(q)-c)\to0,
 \qquad0\le\inf J_n(q)-c\le a_n-c\to0.
\end{equation}
Closed endpoint inclusion is immaterial. The reflected statement holds on the left.
\item If $c$ is a \emph{zero-mass contact}, meaning $Z(c-,t)=Z(c+,t)$ and neither side has a positive fibre at that level, then crossing descendants $J_n(q)$ at the selected times $q\in S_n\subset[0,h_n]$ satisfy
\[
 \sup_{q\in S_n}\max\{|\inf J_n(q)-c|,|\sup J_n(q)-c|\}\longrightarrow0,
\]
with $\sup\varnothing:=0$.
\end{enumerate}
\end{lemma}
\begin{proof}
For \textup{(i)}, failure would give $q_n\le h_n\downarrow0$ and fixed regular labels $b\in I$, $d\notin\overline I$ in the same descendant fibre. Then
\[
 |Z(d,t)-Z(b,t)|\le2V_*q_n\to0,
\]
contradicting maximality of $I$.

For \textup{(ii)}, containment of $a_n$ and absence of labels below $c$ give the left endpoint bound. If the right endpoint stays beyond $c+\epsilon$ at some $q_n\in S_n$, choose a fixed regular $d\in(c+\epsilon/3,c+2\epsilon/3)$. Equality of the positions of $d,a_n$ at $t+q_n$ gives
\[
 |Z(d,t)-Z(a_n,t)|\le2V_*h_n\to0.
\]
Hence $Z(d,t)=Z(c+,t)$, producing a positive right fibre. The left-sided assertion uses $a_n\uparrow c$ and fixed $d<c$ in the putative extension, producing a positive left fibre. Only $q_n\le h_n$ is used, so these bounds are uniform over arbitrary $S_n$.

For \textup{(iii)}, a fixed right extension contains a fixed regular $d>c$ and regular $b_n<c$ with $b_n\uparrow c$. Their equality at times $t+q_n\to t$ yields $Z(d,t)=Z(c-,t)=Z(c+,t)$, contradicting right exposure. A fixed left extension, compared with $b_n>c$ decreasing to $c$, similarly contradicts left exposure.
\end{proof}

\begin{lemma}[Trace moduli and fibre averages]\label{lem:moduli}
At a finite one-sided position trace, use the local BV representative from \textup{\cref{lem:finite-trace-localization}} to define
\begin{align}
 \omega_c^+(d)&=\esssup_{c<a<c+d}|U_\tau(a)-U_\tau(c+)|,\label{eq:omegaplus}\\
 \omega_c^-(d)&=\esssup_{c-d<a<c}|U_\tau(a)-U_\tau(c-)|.
\end{align}
Each modulus on an available side tends to zero as $d\downarrow0$. At a label endpoint only the interior modulus is used, and only when the position trace is finite. A physical velocity on a fibre contained on the corresponding side within label distance $d$ differs from the trace by at most that modulus. If a descendant contains a fixed positive fibre $I$ and adds mass $e$, with $E$ denoting the added labels and $|E|=e$, then
\begin{equation}\label{eq:positive-modulus}
 \left|\mean{I\cup E}{U_\tau}-\mean{I}{U_\tau}\right|
 \le\frac{2R|E|}{|I|+|E|}\le\frac{2Re}{|I|}.
\end{equation}
If the two traces agree at $c$, use the maximum of the two one-sided moduli for any two-sided shrinking fibre.
\end{lemma}
\begin{proof}
The trace moduli follow from \cref{lem:BV}, and fibre averages lie between essential bounds. The bounds also hold on singleton regular labels at almost every time. Subtracting the average over $I$ gives numerator $\int_E(U_\tau-\mean{I}{U_\tau})$, of absolute value at most $2R|E|$, proving \eqref{eq:positive-modulus}.
\end{proof}

\paragraph{Contact continuation and exceptional times.}

\begin{lemma}[Continuation of coinciding traces and order of physical velocities at contact]\label{lem:funnel}
If $Z(c-,t)=Z(c+,t)=z$ at a positive time, then for all $r\ge t$
\begin{equation}\label{eq:funnel}
 Z(c-,r)=Z(c+,r)=:\chi_{c,t}(r),
\end{equation}
and $\chi_{c,t}$ is $V_*$-Lipschitz. At a structural contact time the adjacent physical velocities satisfy $p_Z(c+,t)\le p_Z(c-,t)$, using fibre averages on positive sides and exposed traces on the others.
\end{lemma}
\begin{proof}
Let $a\uparrow c$, $b\downarrow c$ through regular labels in
\[
 0\le Z(b,r)-Z(a,r)\le(r/t)[Z(b,t)-Z(a,t)].
\]
This proves \eqref{eq:funnel}; the $V_*$-Lipschitz bound passes to the traces. At a structural time, the same label limit through Lebesgue points of the physical velocity in the Oleinik condition gives $f(p_+)\le f(p_-)$ and hence $p_+\le p_-$.
\end{proof}

\begin{proposition}[Cluster formation at a strict zero-mass contact]\label{prop:bare-birth}
At a structural time $t>\tau$, suppose $c$ is a zero-mass contact and
$p_Z(c-,t)>p_Z(c+,t)$. There is $h_0>0$ such that $Z(t+h)$ has a positive fibre crossing $c$ for every $0<h<h_0$. A fixed cut can have a strict zero-mass contact at most once after $\tau$. For fixed $\tau$, all such times form an at most countable set.
\end{proposition}
\begin{proof}
By \cref{lem:Dini}, the physical velocities are $\ell=U_\tau(c-)>r=U_\tau(c+)$. Choose $\epsilon>0$ with $\lambda=f(\ell-\epsilon)-f(r+\epsilon)>0$. If no fibre crosses $c$ at times $t+h_n\downarrow t$, the no-splitting property excludes crossing at all intermediate times. Choose regular $a_n<c<b_n$ tending to $c$ with $0\le g_n(0)<\lambda h_n/2$, where $g_n(q):=Z(b_n,t+q)-Z(a_n,t+q)$. By \cref{lem:localization,lem:moduli}, for all large $n$ and a.e. $q\in(0,h_n)$,
\[
 p_Z(a_n,t+q)\ge\ell-\epsilon,\qquad
 p_Z(b_n,t+q)\le r+\epsilon,\qquad g_n'(q)\le-\lambda.
\]
Thus $0\le g_n(h_n)\le g_n(0)-\lambda h_n<0$, a contradiction.

A crossing positive fibre persists, so a given cut cannot have another zero-mass contact. Every strict contact cut is a jump of $U_\tau$. For a compact interval $J\Subset(0,M)$ and $j\ge1$,
\[
 \#\{c\in\operatorname{int}J:|U_\tau(c-)-U_\tau(c+)|\ge1/j\}
 \le j\,\TV_J(U_\tau).
\]
A countable exhaustion by such interiors and the union over $j$ therefore give only countably many jump cuts; each contributes at most one strict contact time. No global BV bound is needed.
\end{proof}
We call a time $t>\tau$ \emph{contact regular} if it is structural and no strict zero-mass contact occurs. These times form a full-measure set determined by the flow and the fixed restart. The preceding countable union is taken before choosing any cut, label sequence, time increment or spatial penalty. At a zero-mass contact at such a time, the order of physical velocities at contact and the preceding proposition give
\begin{equation}\label{eq:equal-bare}
 p_Z(c-,t)=p_Z(c+,t)=U_\tau(c-)=U_\tau(c+)=:r_c.
\end{equation}

\begin{lemma}[Short-time expansion at a zero-mass contact]\label{lem:bare-speed}
At a contact-regular zero-mass contact,
\begin{equation}\label{eq:funnel-error}
 \sup_{0<q\le h}\frac{|\chi_{c,t}(t+q)-z-qf(r_c)|}{q}\to0
 \quad(h\downarrow0).
\end{equation}
\end{lemma}
\begin{proof}
Given any $h_n\downarrow0$, choose regular $a_n\downarrow c$ with
\[
 |Z(a_n,t)-z|+|Z(a_n,t+h_n)-\chi_{c,t}(t+h_n)|<h_n^2.
\]
Partition the intermediate times into those when the descendant stays to the right of $c$ and those when it crosses $c$. Parts \textup{(ii)} and \textup{(iii)} of \cref{lem:localization} place all these fibres in a shrinking two-sided neighbourhood. By \eqref{eq:equal-bare} and \cref{lem:moduli}, there are $\epsilon_n\to0$ for this sequence such that
\[
 \begin{aligned}
 \esssup_{0<q<h_n}|p_Z(a_n,t+q)-r_c|&\le\epsilon_n,\\
 |\chi_{c,t}(t+h_n)-z-h_nf(r_c)|&\le h_n^2+C_fh_n\epsilon_n=o(h_n).
 \end{aligned}
\]
The second bound follows by integration and the endpoint errors. This gives the right derivative of $\chi$ along every sequence, hence \eqref{eq:funnel-error}.
\end{proof}

\subsection[Short-time \texorpdfstring{expansions}{expansions} of \texorpdfstring{trajectories}{trajectories}]{Short-time expansions of trajectories}\label{sec:right-response}\label{sec:capture}

The expansions below hold for arbitrary joint label--time sequences, including capture intervals whose relative length tends to zero.

At an exposed right trace $z=Z(c+,t)$, a positive left fibre at $z$ defines the incoming trajectory $\chi$; at a zero-mass contact use \eqref{eq:funnel}. A positive position gap excludes capture for small increments. At $0+$ only a finite interior position trace is used.

For the right-sided argument, let $a_n>c$ be regular labels with $a_n\downarrow c$, and let $h_n\downarrow0$.

\begin{lemma}[First capture and continuation]\label{lem:capture}
When an incoming trajectory is present, define
\begin{equation}\label{eq:theta}
 d_n(q)=Z(a_n,t+q)-\chi(t+q)\ge0,\quad
 \theta_n=\inf\{q\in[0,h_n]:d_n(q)=0\},
\end{equation}
with $\theta_n=+\infty$ when the set is empty. Then
\[
 \theta_n<\infty\quad\Longrightarrow\quad
 d_n(q)=0\quad(\theta_n\le q\le h_n).
\]
Before capture, $\inf I_Z(a_n,t+q)\ge c$ for $0\le q\le h_n$ with $q<\theta_n$.
\end{lemma}
\begin{proof}
Continuity makes a nonempty zero set of $d_n$ closed, so its infimum is attained. Contact with a positive incoming fibre persists by the no-splitting property. At a zero-mass contact, equality with $\chi$ makes $Z$ constant between $a_n$ and labels arbitrarily close to $c$; these positive intervals persist and their trace remains $\chi$. Thus equality persists after first capture in both cases.

A crossing fibre before capture would contain the old incoming fibre, or both sides of the zero-mass contact, giving an earlier zero of $d_n$. Hence \cref{lem:localization}\textup{(ii)} applies on $[0,\min(h_n,\theta_n))$; after capture the path equals $\chi$.
\end{proof}

When no incoming trajectory is present, we set $\theta_n=+\infty$.

For the following expansions, fix a contact-regular time $t>\tau$ and the refined full-measure set of regular labels.

\begin{proposition}[Right-sided first-capture expansion]\label{thm:right-response}
Let $a_n>c$, $a_n\downarrow c$, $h_n\downarrow0$, and suppose $z=Z(c+,t)$ is finite. If a positive right fibre exists with physical velocity $\pi$, then
\begin{equation}\label{eq:one-branch}
 Z(a_n,t+h_n)=z+h_nf(\pi)+e_n,\quad |e_n|\le h_n\epsilon_n,\quad\epsilon_n\to0.
\end{equation}
Otherwise let $r=p_Z(c+,t)=\underline D_+Q_U(c)$. Let $\ell$ be the physical velocity of a positive incoming left fibre at $z$ when it exists, and set $\ell=r$ otherwise. Then
\begin{equation}\label{eq:right-branch}
 Z(a_n,t+h_n)=\max\{z+h_nf(\ell),\ Z(a_n,t)+h_nf(r)\}+e_n,
 \qquad |e_n|\le h_n\epsilon_n,\quad\epsilon_n\to0.
\end{equation}
The error bound holds for every such joint label--time sequence, including sequences for which capture by $\chi$ occurs after a vanishing fraction of the time increment.
\end{proposition}
\begin{proof}
\emph{Positive fibres.} If the right fibre is positive, $a_n$ lies in a fixed interval $I$ for all large $n$. Its descendant adds at most $\eta_I(h_n)\to0$ mass by \eqref{eq:etaI}, and hence
\[
 |p_Z(a_n,t+q)-\pi|\le 2R\eta_I(h_n)/|I|\quad\text{a.e. }q\in(0,h_n).
\]
Integrating gives \eqref{eq:one-branch} with error at most
$2RC_fh_n\eta_I(h_n)/|I|$.

\emph{Approximation of the exterior and incoming branches.} Suppose instead that the right trace is exposed, and set $d_n=Z(a_n,t)-z\ge0$. Before first capture, localization bounds the descendant fibres within a label radius $\epsilon_n^{\rm lab}\to0$. The averaging moduli give
\begin{equation}\label{eq:pre-error}
 \sup_{0\le q\le\min(h_n,\theta_n)}
 |Z(a_n,t+q)-z-d_n-qf(r)|
 \le C_fh_n\omega_c^+(\epsilon_n^{\rm lab})=o(h_n).
\end{equation}
When capture occurs, continuity from below extends this position estimate to $q=\theta_n$.

If the incoming trajectory is that of a positive left fibre $I^-$, it satisfies
\begin{equation}\label{eq:incoming-error}
 \sup_{0\le q\le h_n}|\chi(t+q)-z-qf(\ell)|
 \le 2RC_fh_n\eta_{I^-}(h_n)/|I^-|=o(h_n).
\end{equation}
For an incoming trajectory at a zero-mass contact, \eqref{eq:funnel-error} gives the corresponding estimate with $\ell=r$. When a positive incoming fibre is present, the order of physical velocities at contact in \cref{lem:funnel} gives $\ell\ge r$.

\emph{Selection of the terminal branch.} If no capture occurs, \eqref{eq:pre-error} gives the terminal position $B_n=Z(a_n,t)+h_nf(r)$ up to $o(h_n)$. The incoming trajectory remains to its left, so its affine approximation $A_n=z+h_nf(\ell)$ satisfies $A_n\le B_n+o(h_n)$. If no incoming trajectory is present, $\ell=r$ and $A_n\le B_n$ directly. A positive position gap excludes capture when $2V_*h_n$ is smaller than the gap; when $V_*=0$, all paths are stationary.

If a positive incoming fibre captures the label, equality at $\theta_n$ in \eqref{eq:pre-error} and \eqref{eq:incoming-error} gives
\begin{equation}\label{eq:meeting-balance}
 d_n=\theta_n(f(\ell)-f(r))+o(h_n).
\end{equation}
The two trajectories agree after capture, giving terminal position $A_n+o(h_n)$. Moreover,
\[
 A_n-B_n=(h_n-\theta_n)(f(\ell)-f(r))+o(h_n)\ge-o(h_n).
\]
Thus $A_n$ is the larger affine contribution up to the stated error. For capture at a zero-mass contact, the two physical velocities agree, \eqref{eq:meeting-balance} gives $d_n=o(h_n)$, and the two contributions coincide to first order.

In each case, the remainder is bounded by $h_n$ times a fixed finite sum of the moduli in \eqref{eq:pre-error}, \eqref{eq:incoming-error} and \eqref{eq:funnel-error}. These moduli tend to zero for the chosen joint sequence, uniformly over intermediate times. Combining the alternatives proves \eqref{eq:right-branch}.
\end{proof}

The left-sided expansion follows by reflection.

\begin{lemma}[Reflection of the Lagrangian evolution]\label{lem:reflect}
Define
\[
 \widehat Z(a,t)=-Z(M-a,t),\quad\widehat U_0(a)=-U_0(M-a),\quad
 \widehat f(r)=-f(-r).
\]
Then $\widehat Z$ is nondecreasing, $\widehat f'(r)=f'(-r)\in[c_f,C_f]$ for $|r|\le R$, and
\[
 P_{\widehat Z}\widehat U_0(a)=-P_ZU_0(M-a),\quad
 \widehat Z_t=\widehat f(P_{\widehat Z}\widehat U_0).
\]
A left contact becomes a right contact, with physical velocities negated and unchanged capture time. The label map preserves measure and error norms.
\end{lemma}
\begin{proof}
The measure-preserving reversal sends $(\alpha,\beta)$ to $(M-\beta,M-\alpha)$ and negates fibre averages and non-flat physical velocities, proving the projection identity. The path equation follows from $\widehat f(-p)=-f(p)$; the reflected positions are nondecreasing. The left gap becomes the right gap with the same zero set, capture time and error norms.
\end{proof}

\begin{corollary}[Left-sided first-capture expansion]\label{cor:left-response}
For regular $a_n<c$, $a_n\uparrow c$, suppose $z=Z(c-,t)$ is finite. A positive left fibre of physical velocity $\pi$ gives $Z(a_n,t+h_n)=z+h_nf(\pi)+o(h_n)$. Otherwise let
$r=p_Z(c-,t)=\overline D_-Q_U(c)$, and let $\ell$ be the physical velocity of a positive incoming right fibre at $z$, or $r$ if none exists. Then
\begin{equation}\label{eq:left-branch}
 Z(a_n,t+h_n)=\min\{z+h_nf(\ell),\ Z(a_n,t)+h_nf(r)\}+e_n,
 \qquad |e_n|\le h_n\epsilon_n,\quad\epsilon_n\to0.
\end{equation}
\end{corollary}
\begin{proof}
For $\widehat c=M-c$ and $\widehat a_n=M-a_n\downarrow\widehat c$, the reflected right physical velocities are $\widehat r=-r$, $\widehat\ell=-\ell$, and the trace position is $\widehat z=-z$. Apply \cref{thm:right-response} to the reflected evolution, multiply by $-1$, and use
\[
 -\max\{-A,-B\}=\min\{A,B\},\qquad\widehat f(-p)=-f(p).
\]
The remainder has the same absolute bound.
\end{proof}

\begin{lemma}[Expansions at finite endpoint position traces]\label{lem:end-response}
Suppose $Z(0+,t)$ is finite. At $0+$, for regular $a_n\downarrow0$, a positive endpoint fibre $I=(0,\beta)$ gives
\[
 Z(a_n,t+h_n)=Z(0+,t)+h_nf(p_I)+o(h_n).
\]
Only if no such fibre exists, put $r_0=U_\tau(0+)=\underline D_+Q_U(0)$ and obtain
\begin{equation}\label{eq:end-exposed}
 Z(a_n,t+h_n)=Z(a_n,t)+h_nf(r_0)+o(h_n).
\end{equation}
When $Z(M-,t)$ is finite, the corresponding $M-$ formulas use a positive left endpoint fibre or, only in the exposed case, $r_M=U_\tau(M-)=\overline D_-Q_U(M)$. No expansion is asserted at an infinite position endpoint. All errors are bounded by $h_n\epsilon_n$ with $\epsilon_n\to0$.
\end{lemma}
\begin{proof}
Finite-position physical velocity traces are supplied by \cref{lem:BV}. On a positive endpoint fibre, \eqref{eq:etaI} and \eqref{eq:positive-modulus} bound the physical velocity error by $2R\eta_I(h_n)/|I|$; its time integral is at most $2RC_fh_n\eta_I(h_n)/|I|=o(h_n)$.

At an exposed $0+$, \cref{lem:localization}\textup{(ii)} applies with $c=0$ and $S_n=[0,h_n]$: no fibre contains labels below zero. Thus, for some $\gamma_n\to0$, its descendants satisfy $I_Z(a_n,t+q)\cap(0,M)\subset(0,\gamma_n)$ throughout the time window. The tower property and \cref{lem:moduli} bound the position error by $C_fh_n\omega_0^+(\gamma_n)=o(h_n)$. At an exposed finite $M-$, the left-sided localization gives $I_Z(a_n,t+q)\cap(0,M)\subset(M-\gamma_n,M)$ and error $C_fh_n\omega_M^-(\gamma_n)=o(h_n)$. Only the interior branch occurs. The exposed Dini identities follow from \cref{lem:Dini}; positive endpoint fibres use their fibre averages.
\end{proof}

\subsection[Comparison of \texorpdfstring{fibre averages}{fibre averages} and \texorpdfstring{physical velocity traces}{physical velocity traces}]{Comparison of fibre averages and physical velocity traces}\label{sec:comparison-times}\label{sec:branch-comparisons}

A spatial penalization localizes near maximizers while preserving level intervals. Compression and the first-capture expansions then give the estimate for the maximal position gap.

Let $X,Y\in\K$ be the Lagrangian maps of two entropy solutions with equal mass and initial physical velocities $U_0,W_0$. First assume $d_0=\|X_0-Y_0\|_\infty<\infty$. Finite displacement gives
\begin{equation}\label{eq:comparison-finite-gap}
 |X(a,t)-Y(a,t)|\le d_0+2V_*T\qquad(0\le t\le T)
\end{equation}
on a common full-measure label set. Set
\[
 \delta=\|U_0-W_0\|_\infty,\quad
 Q_U(a)=\int_0^aU_0,\quad Q_W(a)=\int_0^aW_0.
\]
Fix a common structural restart $\tau>0$ and a common refined set of regular labels. The short-time expansions use the restart profiles, whereas the physical-velocity comparison below uses $Q_U,Q_W$ and the initial discrepancy $\delta$.

Let $\mathcal B^X_\tau,\mathcal B^Y_\tau$ be the at most countable strict zero-mass contact time sets from \cref{prop:bare-birth}, and define
\[
 \mathcal T_\tau=(\mathcal G_X\cap\mathcal G_Y)\setminus
 (\mathcal B^X_\tau\cup\mathcal B^Y_\tau),\qquad t>\tau.
\]
This full-measure set is fixed before choosing labels, increments or spatial penalties. At its times, the short-time expansions of \cref{sec:right-response} apply to both flows at internal cuts and finite position endpoints. The local BV constants may depend on the restart and the neighbourhood of the cut.

\begin{lemma}[Order of physical velocities on common intervals and exposed traces]\label{lem:response-order}
If a positive interval $K$ is a prefix of an $X$ fibre and a suffix of a $Y$ fibre, their physical velocities satisfy
\begin{equation}\label{eq:overlap-order}
 p_X\le\mean{K}{U_0}\le\mean{K}{W_0}+\delta\le p_Y+\delta.
\end{equation}
Here $Q=Q_U$ for an $X$ fibre and $Q=Q_W$ for a $Y$ fibre.
For a positive fibre starting at $c$, $p_I\le\underline D_+Q(c)$; for one ending at $c$, $\overline D_-Q(c)\le p_I$. Corresponding Dini slopes for the two data sets differ by at most $\delta$ in the required direction.
\end{lemma}
\begin{proof}
Prefix and suffix compression at source time zero give the outer inequalities in \eqref{eq:overlap-order}; $U_0\le W_0+\delta$ gives the middle one. Moreover,
\[
 \left|\frac{Q_U(c+h)-Q_U(c)}h
       -\frac{Q_W(c+h)-Q_W(c)}h\right|\le\delta,
\]
and the same bound holds for left quotients. Taking matching lower or upper limits and applying \cref{lem:Dini} proves the remaining assertions, with error independent of interval length.
\end{proof}

For $0<\varepsilon<1$ define
\begin{equation}\label{eq:comparison-penalty}
 T_\varepsilon(x)=x-\varepsilon\sqrt{1+x^2},\qquad
 F_\varepsilon(a,t)=T_\varepsilon(X(a,t))-Y(a,t).
\end{equation}
Since
\[
 1-\varepsilon\le T_\varepsilon'(x)
 =1-\varepsilon\frac{x}{\sqrt{1+x^2}}\le1+\varepsilon,
\]
$T_\varepsilon$ preserves the level intervals of $X$; only the comparison functional is transformed, not the evolution equation. Put
\begin{equation}\label{eq:D}
 D_\varepsilon(t)=\esssup_{0<a<M}F_\varepsilon(a,t),\qquad
 D_+(t)=\esssup_{0<a<M}(X(a,t)-Y(a,t)).
\end{equation}
A one-sided trace, including $0+$ and $M-$, is \emph{active} if both positions are finite and its penalized gap is $D_\varepsilon(t)$.

\begin{lemma}[Localization of near maximizers]\label{lem:active}
For fixed $\varepsilon,T$, $D_\varepsilon$ is finite and $(2+\varepsilon)V_*$-Lipschitz on $[0,T]$. For every fixed $t\in[0,T]$ and sequence $(a_n)$ of common regular labels,
\[
 F_\varepsilon(a_n,t)\to D_\varepsilon(t)
 \quad\Longrightarrow\quad
 \sup_n\bigl(|X(a_n,t)|+|Y(a_n,t)|\bigr)<\infty.
\]
Pairwise distinct near-maximizing sequences exist, and each has a strictly monotone subsequence approaching an active finite trace. No global regulatedness or lower bound for $F_\varepsilon$ is required.
\end{lemma}
\begin{proof}
On the common set of regular labels from \cref{thm:lift}, both maps are continuous in the label at every time and their paths are $V_*$-Lipschitz. Hence $F_\varepsilon(a,t)\le D_\varepsilon(t)$ there, since a strict reverse inequality would persist on a neighbourhood. By \eqref{eq:comparison-finite-gap},
\begin{equation}\label{eq:penalty-tail}
 F_\varepsilon(a,t)
 \le d_0+2V_*T-\varepsilon\sqrt{1+X(a,t)^2}.
\end{equation}
This bounds $D_\varepsilon$ above; a fixed regular label bounds it below. Also,
\begin{equation}\label{eq:penalty-time-Lipschitz}
 |F_\varepsilon(a,t)-F_\varepsilon(a,s)|
 \le(2+\varepsilon)V_*|t-s|.
\end{equation}
Taking essential suprema gives the time bound. For a near-maximizing sequence, \eqref{eq:penalty-tail} gives, eventually,
\[
 \varepsilon\sqrt{1+X(a_n,t)^2}
 \le d_0+2V_*T-D_\varepsilon(t)+1.
\]
Both positions are therefore bounded by \eqref{eq:comparison-finite-gap}. Distinct labels admit a strictly monotone subsequence tending to $c\in[0,M]$. Monotonicity gives finite one-sided position limits, even at label endpoints, and continuity of $T_\varepsilon$ makes the limiting trace active.

Every finite one-sided trace has gap at most $D_\varepsilon(t)$ by approximation through regular labels. The sets
\[
 \{a\text{ regular}:F_\varepsilon(a,t)>D_\varepsilon(t)-1/n\}
\]
have positive measure; successive selection outside finitely many chosen labels yields a distinct near-maximizing sequence. The position bounds exclude infinite endpoint traces.
\end{proof}

\begin{lemma}[Interlacing of interior level intervals]\label{lem:interlace}
At an active right trace $c+\in(0,M)$, let $x=X(c+,t)$ and $y=Y(c+,t)$. For right-continuous representatives,
\begin{equation}\label{eq:global-gap-order}
 T_\varepsilon(X(a,t))-T_\varepsilon(x)\le Y(a,t)-y.
\end{equation}
Let the closures of their level intervals through this interior trace be $I_X=[r,s]$, $I_Y=[p,q]$, with singleton intervals allowed. Then
\begin{equation}\label{eq:interlace}
 p\le r\le c\le q\le s.
\end{equation}
The same statement holds at an active left trace for left-continuous representatives. Exposed domain endpoints, whose level sets may be empty, are treated directly in \textup{\cref{lem:end-order}}.
\end{lemma}
\begin{proof}
The corresponding one-sided representatives satisfy $F_\varepsilon\le D_\varepsilon(t)=T_\varepsilon(x)-y$, giving \eqref{eq:global-gap-order}. On $(r,c)$, $X=x$, hence $Y\ge y$; monotonicity gives $Y=y$ and $p\le r$. On $(c,q)$, $Y=y$, hence $T_\varepsilon(X)\le T_\varepsilon(x)$; strict increase and monotonicity give $X=x$ and $q\le s$. This argument applies to the right-continuous representatives at $c+$ and the left-continuous representatives at $c-$, independently of endpoint values.
\end{proof}

When the overlap has positive length, interlacing makes it a prefix of the $X$ fibre and a suffix of the $Y$ fibre, allowing \cref{lem:response-order} to compare their physical velocities. Exposed physical velocity traces are compared through the corresponding one-sided Dini slopes.

\paragraph{Right, left, and endpoint comparisons.}

For the contact-order identities and short-time comparisons below, fix a common contact-regular time $t\in\mathcal T_\tau$.

At an exposed right side, write $r_Z^+$ for the exterior physical velocity and $\ell_Z^+$ for the incoming left-fibre physical velocity, setting $\ell_Z^+=r_Z^+$ if that fibre is absent. Contact order gives $a_Z^+=\max\{\ell_Z^+,r_Z^+\}=\ell_Z^+$.

\begin{lemma}[Order of physical velocities at an active right trace]\label{lem:Rcases}
At an active interior right trace and a common contact-regular time:
\begin{enumerate}[label=\textup{(\roman*)}]
\item if both positive right fibres exist, $\pi_X\le\pi_Y+\delta$;
\item if only $X$ has a positive right fibre, $\pi_X\le a_Y^++\delta$;
\item if neither has one, $r_X^+\le r_Y^++\delta$ and $a_X^+\le a_Y^++\delta$.
\end{enumerate}
A positive right fibre for $Y$ but not for $X$ is impossible.
\end{lemma}
\begin{proof}
If both right fibres are positive, $q>c$ and $s\ge q$. Their overlap $[r,q]$ is an $X$ prefix and a $Y$ suffix; apply \eqref{eq:overlap-order}.

If only $X$ has a positive right fibre, $q=c<s$. For $r=c$,
\[
 \pi_X\le\underline D_+Q_U(c)
 \le\underline D_+Q_W(c)+\delta=r_Y^++\delta\le a_Y^++\delta.
\]
For $r<c$, interlacing gives $p\le r<c=q$. The overlap $[r,c]$ is an $X$ prefix and a suffix of the incoming $Y$ fibre, hence $\pi_X\le\ell_Y^++\delta\le a_Y^++\delta$.

If both right sides are exposed, $s=q=c$; \cref{lem:Dini,lem:response-order} compare the exterior physical velocities. An incoming $X$ fibre $[r,c]$ forces an incoming $Y$ fibre $[p,c]$ with $p\le r<c$. Thus $K=[r,c]$ is an $X$ prefix and a $Y$ suffix, and \eqref{eq:overlap-order} gives $\ell_X^+\le\ell_Y^++\delta$.
Taking maxima gives the asserted order of $a_X^+,a_Y^+$. Without an incoming $X$ fibre, it follows from $a_X^+=r_X^+\le r_Y^++\delta\le a_Y^++\delta$. Finally, $q>c$ forces $s>c$.
\end{proof}

At an exposed left side, use the exterior physical velocity $r_Z^-$ and incoming right-fibre physical velocity $\ell_Z^-$, setting $\ell_Z^-=r_Z^-$ when absent. Left contact order gives $a_Z^-=\min\{\ell_Z^-,r_Z^-\}=\ell_Z^-$.

\begin{lemma}[Order of physical velocities at an active left trace]\label{lem:Lcases}
At an active interior left trace:
\begin{enumerate}[label=\textup{(\roman*)}]
\item if both positive left fibres exist, $\pi_X\le\pi_Y+\delta$;
\item the only one-sided possibility is a positive left fibre for $Y$ but not $X$, and $a_X^-\le\pi_Y+\delta$;
\item if neither has one, $r_X^-\le r_Y^-+\delta$ and $a_X^-\le a_Y^-+\delta$.
\end{enumerate}
\end{lemma}
\begin{proof}
By \eqref{eq:interlace}, $r<c$ implies $p\le r<c$, excluding a positive left $X$ fibre alone. If both fibres are positive, apply \eqref{eq:overlap-order} on their overlap $[r,q]$.

If only $Y$ has a positive left fibre, $p<c=r$. For $q=c$,
\[
 a_X^-\le r_X^-=\overline D_-Q_U(c)
 \le\overline D_-Q_W(c)+\delta\le\pi_Y+\delta.
\]
For $q>c$, interlacing gives an incoming $X$ fibre $[c,s]$ with $q\le s$. The interval $K=[c,q]$ is an $X$ prefix and a $Y$ suffix, so \eqref{eq:overlap-order} gives $a_X^-\le\ell_X^-\le\pi_Y+\delta$.
If both left sides are exposed, their upper-left Dini slopes compare the exterior physical velocities. An incoming $Y$ fibre $[c,q]$ forces an incoming $X$ fibre $[c,s]$ with $s\ge q$. Since $K=[c,q]$ is an $X$ prefix and the whole $Y$ fibre, hence a $Y$ suffix, \eqref{eq:overlap-order} gives $\ell_X^-\le\ell_Y^-+\delta$.
Taking minima yields $a_X^-\le a_Y^-+\delta$. Without an incoming $Y$ fibre, use $a_X^-\le r_X^-\le r_Y^-+\delta=a_Y^-+\delta$.
\end{proof}

\begin{lemma}[Order of physical velocities at finite endpoint position traces]\label{lem:end-order}
At active finite endpoint traces, the cases with only a positive $Y$ endpoint fibre at $0+$, or only a positive $X$ endpoint fibre at $M-$, are excluded. For two positive endpoint fibres, $\pi_X\le\pi_Y+\delta$; for two exposed traces, $r_X\le r_Y+\delta$. The remaining cases give
\[
 \begin{aligned}
  0+,\quad\text{only $X$ positive}:&\qquad \pi_X\le r_Y+\delta,\\
  M-,\quad\text{only $Y$ positive}:&\qquad r_X\le\pi_Y+\delta.
 \end{aligned}
\]
The physical velocities $\pi_X,\pi_Y$ are used only when the corresponding positive fibres exist. The exposed physical velocities are the values supplied by the lower-right Dini slopes at zero and the upper-left Dini slopes at $M$.
\end{lemma}
\begin{proof}
At $0+$ put $x=X(0+,t)$ and $y=Y(0+,t)$. If $Y=y$ on $(0,q)$, \eqref{eq:global-gap-order} and $X\ge x$ give $X=x$ there. Thus its fibre $(0,s)$ has $s\ge q$. The interval $K=(0,q)$ is an $X$ prefix and the whole $Y$ fibre, hence a $Y$ suffix; \eqref{eq:overlap-order} gives $\pi_X\le\pi_Y+\delta$.
If only $X$ has a positive fibre, left-edge compression gives
\[
 \pi_X\le\underline D_+Q_U(0)
 \le\underline D_+Q_W(0)+\delta=r_Y+\delta.
\]
For two exposed endpoints the two Dini slopes are compared directly.

At $M-$ put $x=X(M-,t)$ and $y=Y(M-,t)$. If $X=x$ on $(r,M)$, \eqref{eq:global-gap-order} and $Y\le y$ give $Y=y$ there. Its fibre $(p,M)$ has $p\le r$. The interval $K=(r,M)$ is the whole $X$ fibre, hence an $X$ prefix, and a $Y$ suffix; \eqref{eq:overlap-order} gives $\pi_X\le\pi_Y+\delta$.
If only $Y$ has a positive endpoint fibre, use
\[
 r_X=\overline D_-Q_U(M)
 \le\overline D_-Q_W(M)+\delta\le\pi_Y+\delta.
\]
For two exposed endpoints compare the upper-left slopes. The finite-position physical velocities exist by \cref{lem:end-response}; neither endpoint argument uses reflection.
\end{proof}

Since $x\mapsto\sqrt{1+x^2}$ is $1$-Lipschitz, for $|v|\le V_*$,
\begin{equation}\label{eq:penalty-affine-error}
 |T_\varepsilon(z+hv)-T_\varepsilon(z)-hv|
 \le\varepsilon V_*h.
\end{equation}
Additional errors are multiplied by at most $1+\varepsilon$, and
\begin{equation}\label{eq:penalty-minmax}
 T_\varepsilon(\max\{A,B\})=\max\{T_\varepsilon(A),T_\varepsilon(B)\},
 \qquad
 T_\varepsilon(\min\{A,B\})=\min\{T_\varepsilon(A),T_\varepsilon(B)\}.
\end{equation}

\begin{lemma}[Short-time estimate at an active trace]\label{thm:active-response}
Fix $\varepsilon\in(0,1)$ and $t\in\mathcal T_\tau$. Let a finite trace $c\pm$ be active. For every compatible monotone sequence of common regular labels $a_n\to c$ and every $h_n\downarrow0$,
\begin{equation}\label{eq:active-response}
 F_\varepsilon(a_n,t+h_n)
 \le D_\varepsilon(t)+(C_f\delta+\varepsilon V_*)h_n
       +h_n\epsilon_n,\qquad\epsilon_n\to0.
\end{equation}
The assertion includes the finite endpoint traces $0+$ and $M-$.
\end{lemma}
\begin{proof}
The order of physical velocities controls the affine contributions in the short-time expansions; the max/min formulas and the penalty bounds then transfer this control to $F_\varepsilon$.

Write $h=h_n$, $x=X(c\pm,t)$, $y=Y(c\pm,t)$ and $D=D_\varepsilon(t)=T_\varepsilon(x)-y$. For $p,q\in[-R,R]$, $p\le q+\delta$ implies $f(p)-f(q)\le C_f\delta$ by monotonicity and the Lipschitz bound. Thus finite branch bases satisfy
\begin{equation}\label{eq:penalty-branch-pair}
 T_\varepsilon(z+hf(p))-(w+hf(q))
 \le T_\varepsilon(z)-w+h(C_f\delta+\varepsilon V_*).
\end{equation}
The joint-sequence remainders remain $o(h)$ after transformation. It suffices to compare the affine branches.

At a right trace, two positive fibres give branches $x+hf(\pi_X)$ and $y+hf(\pi_Y)$, compared by \cref{lem:Rcases}(i). If only $X$ has a positive right fibre, the $Y$ maximum is at least $y+hf(a_Y^+)$, also when no incoming fibre is present. Apply \cref{lem:Rcases}(ii) and \eqref{eq:penalty-branch-pair}.

If both right sides are exposed, \cref{lem:Rcases}(iii) pairs the two incoming physical velocities and the two exterior physical velocities. Their bases have gaps
\[
 T_\varepsilon(x)-y=D,\qquad
 T_\varepsilon(X(a_n,t))-Y(a_n,t)\le D.
\]
Apply \eqref{eq:penalty-branch-pair} to each pair, then \eqref{eq:penalty-minmax} and
\[
 \max\{A_1,A_2\}-\max\{B_1,B_2\}
 \le\max\{A_1-B_1,A_2-B_2\}.
\]
Interlacing excludes the other one-fibre case.

At a left trace, compare two positive fibres by \cref{lem:Lcases}(i). If only $Y$ has a positive left fibre, the transformed $X$ minimum is at most $T_\varepsilon(x+hf(a_X^-))$; compare with $y+hf(\pi_Y)$ by \cref{lem:Lcases}(ii). For two exposed left sides, \cref{lem:Lcases}(iii) pairs the incoming and exterior contributions with the same two base-gap bounds, and
\[
 \min\{A_1,A_2\}-\min\{B_1,B_2\}
 \le\max\{A_1-B_1,A_2-B_2\}
\]
and \eqref{eq:penalty-minmax} give the estimate. Interlacing excludes a positive left $X$ fibre alone.

At finite $0+$, \cref{lem:end-response} gives single branches. For two positive fibres, only a positive $X$ fibre, and two exposed endpoints, their base gaps respectively satisfy
\[
 T_\varepsilon(x)-y=D,\qquad
 T_\varepsilon(x)-Y(a_n,t)\le D,\qquad
 T_\varepsilon(X(a_n,t))-Y(a_n,t)\le D.
\]
The middle bound uses $Y(a_n,t)\ge y$. Apply the corresponding orders in \cref{lem:end-order} and \eqref{eq:penalty-branch-pair}.

At finite $M-$, the cases are two positive fibres, only a positive $Y$ fibre, and two exposed endpoints, with base gaps
\[
 T_\varepsilon(x)-y=D,\qquad
 T_\varepsilon(X(a_n,t))-y\le D,\qquad
 T_\varepsilon(X(a_n,t))-Y(a_n,t)\le D.
\]
The middle bound uses $X(a_n,t)\le x$ and monotonicity of $T_\varepsilon$. The upper-endpoint orders in \cref{lem:end-order} complete these comparisons.

Adding the remainders from \cref{thm:right-response,cor:left-response,lem:end-response} proves \eqref{eq:active-response}. No estimate divides by a capture time, so the error remains $o(h)$ when its ratio to $h$ tends to zero. The error $\varepsilon V_*h$ occurs once per compared branch, not per intermediate event.
\end{proof}

\subsection[The stability estimate and uniqueness]{The stability estimate and uniqueness}\label{sec:envelope}

To estimate the growth of the essential supremum, we will choose near-maximizing labels at the future times $t+h_n$. These labels may vary with $h_n$ and may repeat. The next lemma extends the active-trace estimate to arbitrary near-maximizing sequences, using the common all-time continuity set from \cref{lem:alltime-labels,thm:lift}.

\begin{lemma}[Estimate along near-maximizing sequences]\label{thm:all-sequences}
For every fixed $\varepsilon\in(0,1)$, every $t\in\mathcal T_\tau$, every $h_n\downarrow0$, and every sequence of common regular labels satisfying $F_\varepsilon(a_n,t)\to D_\varepsilon(t)$,
\begin{equation}\label{eq:all-sequences}
 \limsup_n\frac{F_\varepsilon(a_n,t+h_n)-D_\varepsilon(t)}{h_n}
 \le C_f\delta+\varepsilon V_*.
\end{equation}
\end{lemma}
\begin{proof}
Suppose a subsequence violates the bound by a fixed positive amount. Infinitely many distinct labels would, by \cref{lem:active}, give a monotone subsequence approaching an active finite trace, contradicting \cref{thm:active-response}. Only local regularity near that finite trace is used.

Otherwise pass to a repeated label $a_*$. Both maps are continuous there at $t$ and $t+h_n$. Choose distinct common regular labels $b_n$, $|b_n-a_*|<1/n$, with
\[
 \sum_{r\in\{t,t+h_n\}}
 \bigl(|X(b_n,r)-X(a_*,r)|+|Y(b_n,r)-Y(a_*,r)|\bigr)
 <\frac{h_n^2}{2+\varepsilon}.
\]
Continuity and density permit each selection outside finitely many previously chosen labels. The two penalized-gap errors sum to less than $h_n^2$. Thus $b_n$ remain near-maximizing and change the quotient by at most $h_n$, contradicting the distinct-label case.
\end{proof}

Write $D^+$ for the upper right Dini derivative with respect to time.

\begin{lemma}[Differentiation of the essential supremum]\label{lem:envelope}
Let $A\subset(0,M)$ have full measure. Suppose $F(a,t)$ is finite on $A$, measurable in $a$, and
\[
 |F(a,t)-F(a,s)|\le L|t-s|\qquad(a\in A).
\]
Assume $D(t)=\esssup_aF(a,t)$ is finite, $F(a,t)\le D(t)$ for every $a\in A$ and every $t$, and a fixed full-measure time set has the following property: for every sequence $h_n\downarrow0$ and $a_n\in A$ with $F(a_n,t)\to D(t)$,
\[
 \limsup_n\frac{F(a_n,t+h_n)-D(t)}{h_n}\le C.
\]
Then $D$ is $L$-Lipschitz and $D^+D(t)\le C$ at those times. Neither a global lower bound for $F$ nor global regulatedness in the label is assumed.
\end{lemma}
\begin{proof}
Taking essential suprema in the increment bound gives $|D(t)-D(s)|\le L|t-s|$. At a specified time choose $h_n\downarrow0$ realizing $D^+D(t)$, and future near maximizers $a_n\in A$ with
\[
 D(t+h_n)-h_n^2\le F(a_n,t+h_n),\qquad
 F(a_n,t)\le D(t).
\]
It follows that
\[
 D(t)\ge F(a_n,t)
 \ge D(t+h_n)-Lh_n-h_n^2
 \ge D(t)-2Lh_n-h_n^2.
\]
The assumed estimate therefore applies, and
\[
 \frac{D(t+h_n)-D(t)}{h_n}
 \le\frac{F(a_n,t+h_n)-D(t)}{h_n}+h_n
\]
gives $D^+D(t)\le C$.
\end{proof}

\begin{proof}[Proof of the comparison and uniqueness assertions in \textup{\cref{thm:main}}]
\emph{Differentiation and integration at fixed parameters.} Fix $T>0$, a common structural restart $0<\tau<T$, and $\varepsilon\in(0,1)$. By \cref{lem:active}, $F_\varepsilon$ meets the hypotheses of \cref{lem:envelope} with $L=(2+\varepsilon)V_*$. Thus \cref{thm:all-sequences} gives
\[
 D^+D_\varepsilon(t)\le C_f\delta+\varepsilon V_*
 \quad\text{for every }t\in\mathcal T_\tau.
\]
Absolute continuity of $D_\varepsilon$ gives, for $t\in[\tau,T]$,
\begin{equation}\label{eq:penalized-integrated-comparison}
 D_\varepsilon(t)\le D_\varepsilon(\tau)
 +(C_f\delta+\varepsilon V_*)(t-\tau).
\end{equation}
The short-time limit has thus been taken and integrated with both $\varepsilon$ and $\tau$ fixed.

\emph{Removal of the spatial penalty and restart.} For each fixed time, $F_\varepsilon\uparrow X-Y$ almost everywhere as $\varepsilon\downarrow0$, and
\[
 D_\varepsilon(t)\uparrow D_+(t).
\]
Indeed, choose a regular label with gap above $D_+(t)-\eta$. Its finite position makes its penalty less than $\eta$ for small $\varepsilon$, so $D_+(t)-2\eta<D_\varepsilon(t)\le D_+(t)$. Sending $\varepsilon\downarrow0$ in \eqref{eq:penalized-integrated-comparison} gives
\[
 D_+(t)\le D_+(\tau)+C_f\delta(t-\tau).
\]
Now choose common structural restarts tending to zero. Finite displacement yields
\[
 D_+(\tau)\le\esssup_a(X_0(a)-Y_0(a))+2V_*\tau.
\]
Send $\tau\downarrow0$, then exchange the original flows to obtain
\begin{equation}\label{eq:stability}
 \|X(t)-Y(t)\|_\infty
 \le\|X_0-Y_0\|_\infty+C_ft\|U_0-W_0\|_\infty.
\end{equation}
Since $X(t)=X_{\rho_t}$ and $Y(t)=X_{\widetilde\rho_t}$, \eqref{eq:quantile-Winfty} identifies both position discrepancies with the corresponding normalized mass distances. With $\delta=\delta_0$, this proves \eqref{eq:main-mass-stability}.
The local BV and averaging constants enter only remainders eliminated before the parameter limits; the integrated constant remains $C_f\delta$. For $d_0=\infty$, the estimate has only its extended-value meaning.

\emph{Equal-data uniqueness on a common label set.} For identical data, $X(t)=Y(t)$ almost everywhere at each time. Intersect the equality sets for rational times with the common all-time path set of \cref{thm:lift}. Lipschitz continuity gives equality for all $t\ge0$ on this fixed full-measure label set. The push-forward formulas in \eqref{eq:entropy-representation}, with the same initial physical velocity $U_0=W_0$, give $\rho_t^X=\rho_t^Y$ and $\qmeas_t^X=\qmeas_t^Y$ at every time; their densities agree $\rho_t\dd t$-almost everywhere, proving uniqueness in \cref{def:entropy}.
\end{proof}
\section[\texorpdfstring{Kinetic energy dissipation and the Lagrangian semigroup}{Kinetic energy dissipation and the Lagrangian semigroup}]{Kinetic energy dissipation and the Lagrangian semigroup}\label{sec:consequences}

The kinetic energy of the unique flow is nonincreasing, and the flow can be restarted from any time. Let $p_X(t)=P_{X(t)}U_0$ be the canonical Lagrangian physical velocity of the unique Lagrangian flow. All label norms below are taken on $(0,M)$.

\subsection[\texorpdfstring{Kinetic energy dissipation and convex entropies}{Kinetic energy dissipation and convex entropies}]{Kinetic energy dissipation and convex entropies}\label{sec:energy-entropies}

\begin{proposition}[Kinetic energy identity and right continuity]\label{thm:energy}
For $0\le s\le t$,
\begin{equation}\label{eq:energy-identity}
 p_X(t)=P_{X(t)}p_X(s),\qquad
 \|p_X(s)\|_2^2-\|p_X(t)\|_2^2=\|p_X(s)-p_X(t)\|_2^2.
\end{equation}
The map $p_X$ is strongly right-continuous in $L^2$ at every $s\ge0$, and
\begin{equation}\label{eq:L2-right-velocity}
 \lim_{h\downarrow0}
 \left\|\frac{X(s+h)-X(s)}h-f(p_X(s))\right\|_2=0.
\end{equation}
If $I$ is a positive fibre at time $s>0$, then, for every regular label $a$ in its interior, the right time derivative of $X(a,\cdot)$ at $s$ equals $f(p_I)$.
\end{proposition}
\begin{proof}
\emph{Projection loss.} The no-splitting property gives $\mathcal H_{X(t)}\subset\mathcal H_{X(s)}$. The tower property yields $p_X(t)=P_{X(t)}p_X(s)$, and orthogonality proves \eqref{eq:energy-identity}.

\emph{Strong right continuity and the right derivative.} Since $p_X(s)\in\mathcal H_{X(s)}$ and $X(s+h)\to X(s)$ in $L^2$, \cref{lem:projection-approximation} gives strong right continuity at every $s$, including $s=0$ by compatibility. The evolution equation then gives
\[
 \left\|\frac{X(s+h)-X(s)}h-f(p_X(s))\right\|_2
 \le\frac{C_f}{h}\int_s^{s+h}
 \|p_X(r)-p_X(s)\|_2\dd r\longrightarrow0.
\]
\emph{Pointwise transport velocity on positive fibres.} On a positive fibre $I$ at time $s$, the no-splitting property makes every right difference quotient constant on $I$. Its constant value therefore converges to $f(p_I)$ by \eqref{eq:L2-right-velocity}. The canonical position representative is constant on the interior of $I$ at every $t\ge s$, also when the fibre is formed at time $s$.
\end{proof}

\begin{proposition}[Convex entropies]\label{prop:entropies}
For every continuous convex $\eta:[-R,R]\to\R$, the solution satisfies
\begin{equation}\label{eq:entropy-dist}
 \partial_t(\eta(u^\sharp)\rho)+
 \partial_x(\eta(u^\sharp)f(u^\sharp)\rho)\le0
\end{equation}
in distributions, with the prescribed initial entropy trace in the narrow topology. In particular, the total entropy $\int_\R\eta(u^\sharp(x,t))\dd\rho_t(x)$ is nonincreasing in $t$.
\end{proposition}
\begin{proof}
\emph{Conditional Jensen's inequality and time summation.} Fix $T>0$, $0\le\phi\in C_c^1(\R\times[0,T))$, and put $h=T/N$, $t_j=jh$, $p_j=p_X(t_j)$, $X_j=X(t_j)$. Conditional Jensen's inequality and \eqref{eq:energy-identity} give
\[
 \eta(p_{j+1})\le P_{X_{j+1}}\eta(p_j).
\]
Testing against the nonnegative $\sigma(X_{j+1})$-measurable function $\phi(X_{j+1},t_{j+1})$, and setting $E_j=\int_0^M\eta(p_j)\phi(X_j,t_j)\dd a$, yields
\[
 \begin{aligned}
 E_{j+1}-E_j
 &\le\int_0^M\eta(p_j)
 [\phi(X_{j+1},t_{j+1})-\phi(X_j,t_j)]\dd a\\
 &=\int_{t_j}^{t_{j+1}}\!\int_0^M
 \eta(p_j)[\phi_t(X,t)+f(p_X(t))\phi_x(X,t)]\dd a\dd t,
 \end{aligned}
\]
by the chain rule along the Lagrangian paths. With the left-endpoint time interpolant $p_N^-(t)=p_j$ on $[t_j,t_{j+1})$, summation, $E_N=0$ and $p_0=U_0$ give
\[
 \int_0^T\!\int_0^M
 \eta(p_N^-)[\phi_t(X,t)+f(p_X(t))\phi_x(X,t)]\dd a\dd t
 +\int_0^M\eta(U_0)\phi(X_0,0)\dd a\ge0.
\]
\emph{Convergence from the projection loss.} Writing $e(t)=\|p_X(t)\|_2^2$, the exact projection loss gives
\[
 \begin{aligned}
 \int_0^T\|p_N^-(t)-p_X(t)\|_2^2\dd t
 &=\sum_{j=0}^{N-1}\int_{t_j}^{t_{j+1}}[e(t_j)-e(t)]\dd t\\
 &\le h[e(0)-e(T)]\longrightarrow0.
 \end{aligned}
\]
\begingroup
To pass the continuous convex function $\eta$ through this limit, let
\[
 \omega_\eta(\delta)=\sup\{|\eta(r)-\eta(s)|:
 r,s\in[-R,R],\ |r-s|\le\delta\}.
\]
Splitting according to whether $|p_N^--p_X|\le\delta$ gives, for every $\delta>0$,
\[
 \begin{aligned}
 \int_0^T\!\int_0^M|\eta(p_N^-)-\eta(p_X)|\dd a\dd t
 &\le MT\omega_\eta(\delta)\\
 &\quad+\frac{2\|\eta\|_{L^\infty([-R,R])}}{\delta^2}
 \int_0^T\|p_N^-(t)-p_X(t)\|_2^2\dd t.
 \end{aligned}
\]
First let $N\to\infty$, then $\delta\downarrow0$. Uniform continuity gives $\eta(p_N^-)\to\eta(p_X)$ in $L^1((0,M)\times(0,T))$. The factor multiplying $\eta(p_N^-)$ in the summed inequality is bounded by $\|\phi_t\|_\infty+V_*\|\phi_x\|_\infty$.
\endgroup
Passing to the limit and using \eqref{eq:pullback} proves \eqref{eq:entropy-dist}, including its initial term.

\emph{Initial trace and total entropy.} For the initial entropy trace, \cref{thm:energy} gives $p_X(t)\to U_0$ in $L^2$, hence $\eta(p_X(t))\to\eta(U_0)$ in $L^1$. On the common all-time path set, finite displacement gives $X(a,t)\to X_0(a)$. For every $\psi\in C_b(\R)$, splitting the product difference and applying dominated convergence on the finite label interval yields
\[
 \eta(p_X(t))\psi(X(t))\longrightarrow\eta(U_0)\psi(X_0)
 \quad\text{in }L^1(0,M).
\]
This proves the narrow initial entropy trace through all positive times. For every $s\le t$, conditional Jensen's inequality and preservation of integrals by $P_{X(t)}$ give $\int_0^M\eta(p_X(t))\dd a\le\int_0^M\eta(p_X(s))\dd a$.
\end{proof}

\subsection[\texorpdfstring{Restarting the flow and the semigroup property}{Restarting the flow and the semigroup property}]{Restarting the flow and the semigroup property}\label{sec:restart-semigroup}

The strong right continuity in \cref{thm:energy} supplies the initial kinetic energy trace when the flow is restarted at an arbitrary time. We verify that the continuation belongs to the same entropy class and then apply uniqueness to obtain the semigroup property.

\begin{theorem}[Lagrangian semigroup]\label{prop:semigroup}
For fixed $M,R,f$, let
\[
 \mathcal A_{M,R}=\{(X,U):X\in\K\subset L^2(0,M),\ U\in L^\infty(0,M),
             \ |U|\le R,\ P_XU=U\}.
\]
The map
\[
 \mathcal S_t(X_0,U_0)=(X(t),p_X(t))
\]
takes $\mathcal A_{M,R}$ into itself and satisfies
$\mathcal S_{t+s}=\mathcal S_t\circ\mathcal S_s$.
\end{theorem}
\begin{proof}
\emph{Admissibility of the restart data.} Fix $s\ge0$ and set $X_s=X(s)$, $U_s=p_X(s)$. The moment bound, $|U_s|\le R$ and $U_s\in\mathcal H_{X_s}$ give $(X_s,U_s)\in\mathcal A_{M,R}$, so \cref{sec:existence} constructs a global entropy solution from these data.

\emph{Entropy admissibility of the continuation.} For $Z(r)=X(s+r)$ and $W(r)=p_X(s+r)$, the tower property gives
\[
 \begin{gathered}
 W(r)=P_{Z(r)}U_s,\qquad
 Z(r)=X_s+\int_0^r f(W(q))\dd q,\\
 \rho_{s+r}=Z(r)_\#\dd a,\qquad
 \qmeas_{s+r}=Z(r)_\#(U_s\dd a).
 \end{gathered}
\]
The pathwise chain-rule argument of \cref{prop:closure}, with the fixed label weights $1$ and $U_s$, supplies both weak equations and their initial terms. The translated transport coefficient $b_s(x,r)=b(x,s+r)$ is jointly Borel, has the same transport-velocity bound and agrees with the transport velocity $\rho_{s+r}\dd r$-almost everywhere. For almost every $r>0$,
\[
 b_s(y,r)-b_s(x,r)
 \le\frac{y-x}{s+r}\le\frac{y-x}{r}\qquad(x<y).
\]
By \cref{thm:energy}, $W(r)\to U_s$ strongly in $L^2$ as $r\downarrow0$. Together with $Z(r)\to X_s$ along the common label paths, the same product estimate as in \cref{lem:constructed-energy} gives, for every $\varphi\in C_b(\R)$,
\[
 \int_0^M\varphi(Z(r))W(r)^2\dd a
 \longrightarrow\int_0^M\varphi(X_s)U_s^2\dd a.
\]
For every $T>0$, choose $G$ on $[0,s+T]$. The continuation has this initial energy-measure trace on the fixed full-measure set $\{0<r<T:s+r\in G\}$, where its physical velocity agrees $\rho_{s+r}$-almost everywhere with the canonical physical velocity. Thus the continuation satisfies \cref{def:entropy} with source data $(X_s,U_s)$.

\emph{Identification by uniqueness.} Uniqueness from \cref{sec:compare} identifies it with the solution from these data, including its Lagrangian physical velocity $P_{Z(r)}U_s$. Hence $\mathcal S_r\mathcal S_s=\mathcal S_{s+r}$; compatibility gives $\mathcal S_0=\mathrm{Id}$.
\end{proof}

\appendix

\section[Quantiles, traces and representatives]{Quantiles, traces and representatives}\label{app:measure}

\begin{proof}[Proof of \textup{\cref{lem:quantile-details}}]
$T\circ X_\mu$ is finite and nondecreasing on a full-measure label set and pushes $\dd a$ to $\nu$. Its finite interior one-sided completion has the same sublevel-set lengths as the quantile of $\nu$, hence agrees with it almost everywhere. No infinite endpoint value is needed.
\end{proof}

\begin{proof}[Proof of \textup{\cref{lem:alltime-labels}}]
\emph{Time regularity of the one-sided traces.} Interior traces are finite by local boundedness of monotone $L^2$ maps. Passing to one-sided limits through the full-measure set where the increment bound holds gives
$|X(c\pm,t)-X(c\pm,s)|\le C|t-s|$ for every interior cut $c$. Letting $c$ approach an endpoint proves the same bound whenever its trace is finite; the two order inequalities force an infinite endpoint trace to retain its sign and value at all times.

\emph{A common continuity-label set.} Remove the countable union of jump labels of $X(\cdot,q)$ over rational $q\in[0,T]$. A jump of size $d>0$ at a remaining label and a time $t$ would give a jump at least $d-2C|q-t|>0$ at a sufficiently close rational time, a contradiction. For $C=0$ all slices agree. This proves simultaneous continuity on the remaining labels.
\end{proof}

\begin{proof}[Proof of \textup{\cref{lem:joint}}]
\emph{Jointly Borel position and physical velocity representatives.} The representative-independent primitive $A(a,t)=\int_0^aX(b,t)\dd b$ satisfies, with $C_T=\sup_{[0,T]}\|X(t)\|_2$,
\begin{equation}\label{eq:primitive-joint-modulus}
 |A(a,t)-A(a',t')|
 \le C_T|a-a'|^{1/2}+\sqrt M\|X(t)-X(t')\|_2.
\end{equation}
It is jointly continuous. For $0<a<M$ and $h_n(a)=\min\{1/n,(M-a)/2\}$, the Borel limit
\[
 X^+(a,t)=\lim_{n\to\infty}
       \frac{A(a+h_n(a),t)-A(a,t)}{h_n(a)}
\]
is finite and equals the right-continuous monotone representative at every time. From now on, $X$ denotes this jointly Borel right-continuous representative. Define the jointly Borel fibre endpoints
\[
 l(a,t)=\int_0^M\one_{\{X(b,t)<X(a,t)\}}\dd b,\qquad
 r(a,t)=\int_0^M\one_{\{X(b,t)\le X(a,t)\}}\dd b.
\]
For $r>l$ set $p(a,t)=[Q_U(r)-Q_U(l)]/(r-l)$; for $r=l$ use a fixed Borel representative of $U_0(a)$. By \cref{lem:PX}, this is $P_{X(t)}U_0$ almost everywhere for every $t$.

\emph{Narrow continuity of the mass and momentum measures.} If $t_n\to t$, then $X(t_n)\to X(t)$ in $L^2(0,M)$. For every $\psi\in C_b(\R)$, continuity gives $\psi(X(t_n))\to\psi(X(t))$ in measure; uniform boundedness on the finite label interval gives convergence in $L^1$. Multiplication by the fixed bounded initial physical velocity $U_0$ preserves this convergence, proving narrow continuity of both push-forwards. For a Borel set $D$, the push-forward formulas also show that $t\mapsto\rho_t(D),\qmeas_t(D)$ are Borel.

\emph{A jointly Borel representative of the physical velocity.} Let $D_n(x)$ be the half-open dyadic interval of length $2^{-n}$ containing $x$, and define
\[
 u_n(x,t)=
 \begin{cases}
 \qmeas_t(D_n(x))/\rho_t(D_n(x)),&\rho_t(D_n(x))>0,\\
 0,&\rho_t(D_n(x))=0.
 \end{cases}
\]
These functions are jointly Borel and bounded by $R$. Their limit where it exists, extended by zero elsewhere, is therefore a bounded Borel function $u^\sharp$. For each fixed $t$, $|\qmeas_t|\le R\rho_t$ gives a density $w_t=d\qmeas_t/d\rho_t$, and $u_n(\cdot,t)$ is its conditional expectation on the dyadic partition. These increasing sigma-algebras generate the Borel sigma-algebra; the upward martingale convergence theorem \cite[Theorem 4.6.8]{Durrett}, applied to the probability measure $\rho_t/M$, gives $u_n(\cdot,t)\to w_t$ almost everywhere and in $L^1(\rho_t)$. This argument applies separately at every time, without a common null set in $x$.

\emph{Identification of the pullback.} Both $p(t)$ and $u^\sharp(X(t),t)$ lie in $\mathcal H_{X(t)}$ and have the same integrals against every bounded function of $X(t)$, by their common push-forward formula for the momentum measure. They therefore coincide in $L^2$, proving \eqref{eq:pullback} at every time.
\end{proof}
\section*{Use of generative AI}

During the preparation of this manuscript, the authors used ChatGPT (OpenAI) to assist with language editing, \LaTeX{} formatting, and the organization, review and clarification of the mathematical exposition. The authors independently assessed and verified the suggestions incorporated into the manuscript and take full responsibility for its content, including the mathematical statements, computations, proofs and references.

\bibliographystyle{abbrv}
\bibliography{references}
\end{document}